\documentclass[10pt, reqno]{amsart}
\usepackage{tikz}
\usetikzlibrary{matrix,arrows,decorations.pathmorphing}
\usepackage{amssymb}

\newtheorem{proposition}{Proposition}[section]
\newtheorem{lemma}[proposition]{Lemma}
\newtheorem{corollary}[proposition]{Corollary}
\newtheorem{theorem}[proposition]{Theorem}
\newtheorem{remark}[proposition]{Remark}

\theoremstyle{definition}

\newcommand{\selabel}[1]{\label{se:#1}}

\newcommand{\eqlabel}[1]{\label{eq:#1}}

\def\<{\leqslant}
\def\>{\geqslant}
\def\a{\alpha}
\def\b{\beta}

\def\g{\gamma}
\def\G{\Gamma}
\def\O{\Omega}

\def\ol{\overline}

\def\t{\triangle}

\def\l{\lambda}
\def\L{\Lambda}
\def\s{\sigma}

\def\ot{\otimes}
\def\ra{\rightarrow}

\date{}

\begin{document}
\title[Representations of the Drinfeld doubles ]{Representations of the Drinfeld doubles of pointed rank one Hopf algebras}
\author{Hua Sun}
\address{School of Mathematical Science, Yangzhou University,
	Yangzhou 225002, China}
\email{huasun@yzu.edu.cn}
\author{Hui-Xiang Chen}
\address{School of Mathematical Science, Yangzhou University,
	Yangzhou 225002, China}
\email{hxchen@yzu.edu.cn}
\author{Yinhuo Zhang}
\address{Department of Mathematics $\&$  Statistics, University of Hasselt, Universitaire Campus, 3590 Diepeenbeek, Belgium
	Yangzhou 225002, China}
\email{yinhuo.zhang@uhasselt.be}
\begin{abstract}
In this paper, we investigate the  representations of the Drinfeld doubles $D(H_{\mathcal{D}})$
of pointed rank one Hopf algebras $H_{\mathcal{D}}$ over an algebraically closed field
$\Bbbk$ of characteristic zero. We provide a complete classification of all finite-dimensional indecomposable $D(H_{\mathcal{D}})$-modules up to isomorphism and explicitly describe the Auslander-Reiten sequences in the category of finite-dimensional $D(H_{\mathcal{D}})$-modules. We show that  $D(H_{\mathcal{D}})$ is of  tame representation type.
\end{abstract}
\maketitle

\section{Introduction}\selabel{1}

Finite-dimensional Hopf algebras play a central role in modern algebra and quantum group theory, particularly through their connections with tensor categories, knot and 3-manifold invariants, and noncommutative geometry. Among them, pointed Hopf algebras of rank one constitute one of the simplest nontrivial families. Despite their apparent simplicity, their structure and representation theory already exhibit many of the key phenomena appearing in the general theory.

The systematic study of pointed Hopf algebras began in the 1990s with the classification program of Andruskiewitsch and Schneider, who reduced the classification problem to the analysis of Nichols algebras. In the rank-one case, the corresponding Nichols algebra is one-dimensional, and the resulting Hopf algebras generalize the classical Taft algebras introduced by E. Taft \cite{Taft}, which provided the first explicit examples of finite-dimensional nonsemisimple, noncommutative Hopf algebras.

The first systematic investigation of finite-dimensional pointed Hopf algebras of rank one was carried out by Krop and Radford \cite{KR}, who classified all such algebras over algebraically closed fields of characteristic zero. The classification in positive characteristic was later completed by Scherotzke \cite{Sche}. These Hopf algebras can also be realized as quotients of Hopf-Ore extensions of group algebras, a framework developed by Wang, You, and Chen. The classification over arbitrary fields was subsequently obtained by Wang, You and Chen \cite{WYC}.

The representation theory of these Hopf algebras has been the subject of intensive study. In \cite{Cib}, Cibils investigated the indecomposable modules over Taft algebras and derived explicit decomposition formulas for tensor products of indecomposable modules. Building on Cibils's results, Chen, Van Oystaeyen, and Zhang determined the Green rings of all Taft algebras \cite{CVOZ}, while Li and Zhang computed the Green rings of generalized Taft algebras \cite{LZ}. The representation theory of finite-dimensional pointed Hopf algebras of rank one was subsequently studied in full generality by Wang, Li, and Zhang \cite{WLZ14, WLZ16}. These works describe not only the categories of indecomposable modules but also the fusion rules governing their tensor product structures.

Another line of research concerns the Drinfeld doubles of pointed rank-one Hopf algebras, which are quasi-triangular and therefore play a fundamental role in the study of braided tensor categories. The representation theory of Drinfeld doubles of Taft algebras was studied in detail by Chen \cite{Ch2, Ch3, Ch4, Ch5}, who described the simple and indecomposable modules, projective covers, and block structures. These results highlight the deep connections between pointed rank-one Hopf algebras, their doubles, and quantum groups at roots of unity. Erdmann et al. investigated the representations and stable Green rings of Drinfeld doubles of generalized Taft algebras in \cite{EGST2006, EGST2019}, while Sun and Chen studied the representations of Drinfeld doubles of Radford Hopf algebras in \cite{SunChen}. The (generalized) Taft algebras are pointed rank-one Hopf algebras of nilpotent type, whereas the Radford algebras are of non-nilpotent type. Krop and Radford studied the representations of the Drinfeld double $D(H_{\mathcal{D}})$ of any finite-dimensional pointed rank one  Hopf algebra $H_{\mathcal{D}}$  over an algebraically closed field of characteristic zero, classifying all simple and projective indecomposable $D(H_{\mathcal{D}})$-modules in the case where the group $G(H_{\mathcal{D}})$ of group-like elements is abelian \cite{KR}. However, a complete classification of all finite-dimensional indecomposable $D(H_{\mathcal{D}})$-modules remains open.

In this article, building on the work of Krop and Radford \cite{KR}, we study the indecomposable representations of the Drinfeld doubles $D(H_{\mathcal{D}})$ of pointed rank-one Hopf algebras $H_{\mathcal{D}}$ over an algebraically closed field $\Bbbk$ of characteristic zero, assuming that $G(H_{\mathcal{D}})$ is abelian.

The paper is organized as follows.  In Section \ref{s2}, we recall the definition of a group datum, the structure of a pointed rank one Hopf algebra  $H_{\mathcal{D}}$, and its Drinfeld double $D(H_{\mathcal{D}})$.

In Section \ref{s3}, we review the simple modules and projective indecomposable modules over $D(H_{\mathcal{D}})$, showing that the Loewy length of $D(H_{\mathcal{D}})$ is three. We divide both the simple and non-simple projective indecomposable modules into the nilpotent and non-nilpotent cases, and provide standard $\Bbbk$-bases for each of these modules.

In Section \ref{s4}, we investigate the finite-dimensional indecomposable $D(H_{\mathcal{D}})$-modules of Loewy length two. Using Auslander-Reiten theory, we describe all such modules and classify them up to isomorphism. Moreover, all Auslander-Reiten sequences in the category of finite-dimensional $D(H_{\mathcal{D}})$-modules are explicitly presented. Finally, we show that $D(H_{\mathcal{D}})$ is of tame representation type.
The main results  of the paper are  as follows:

\textbf{Theorem 4.23}
Assume $m>1$. A complete set of representatives of isomorphism classes of finite dimensional indecomposable $D(H_{\mathcal D})$-modules is given by
$$\left\{\begin{array}{c}
V(l',\l'), P(l,\l), \Omega^{\pm s}V(l,\l),\\
T_s(l,\l), \ol{T}_s(l,\l), M_s(l,\l,\eta)\\
\end{array}\left|\begin{array}{c}
1\<l\<n-1, 1\<l'\<n, s\>1,\\
\l \in I_l, \l'\in I_{l'}, \eta\in\Bbbk^{\times}\\
\end{array}\right.\right\}.$$

\textbf{Theorem 4.30}
Assume $m=1$. A complete set of  representatives  of isomorphism classes of finite dimensional indecomposable $D(H_{\mathcal D})$-modules is given by
$$\left\{\begin{array}{c}
V(l',\l'), P(l,\l),\\ \Omega^{\pm s}V(l,\l),
 W_s(l,\l,\eta)\\
\end{array}\left|\begin{array}{c}
1\<l\<n-1, 1\<l'\<n, s\>1,\\
\l \in I_l, \l'\in I_{l'}, \eta\in\overline{\Bbbk}\\
\end{array}\right.\right\}.$$

Throughout, let $\Bbbk$ be an algebraically closed field with char$\Bbbk=0$ and $\Bbbk^{\times}=\Bbbk\backslash\{0\}$. Unless
otherwise stated, all algebras and Hopf algebras are
defined over $\Bbbk$; all modules are finite dimensional and left modules;
dim and $\otimes$ denote ${\rm dim}_{\Bbbk}$ and $\otimes_{\Bbbk}$,
respectively. Let $\mathbb Z$ denote the set of all integers, ${\mathbb Z}_n:={\mathbb Z}/n{\mathbb Z}$
for an integer $n$, and let $\mathbb{N}$ denote all non-negative integers.
The references \cite{Ka, Mon, Sw} are basic references for the theory of Hopf algebras and quantum groups. The readers may refer to \cite{ARS} for the representation theory of algebras.

\section{Pointed Hopf algebras of rank one and their doubles}\label{s2}

In this section, we recall the pointed Hopf algebras of rank one and  their Drinfeld doubles.

Let $0\neq q\in\Bbbk$. For any integer $n>0$, set $(n)_q=1+q+\cdots+q^{n-1}$. Observe that $(n)_q=n$ when $q=1$ and $(n)_q=\frac{q^n-1}{q-1}$ when $q\neq 1$. Define the $q$-factorial of $n$ by $(0)!_q=1$ and $(n)!_q=(n)_q(n-1)_q\cdots(1)_q$ for $n>0$, see \cite[p.74]{Ka}.

A quadruple $\mathcal{D}=(G, \chi, a, \a)$ is called a {\it group datum} if $G$ is a finite group, $\chi$ is a $\Bbbk$-valued character of $G$, $a$ is a central element of $G$ and $\a\in\Bbbk$ subject to $\chi^n=1$ or $\a(a^n-1)=0$, where $n$ is the order of $\chi(a)$. The group datum $\mathcal{D}$ is of {\it nilpotent type} if $\a(a^n-1)=0$, and it is of {\it non-nilpotent type} if $\a(a^n-1)\neq 0$ and $\chi^n=1$. For any group datum $\mathcal{D}=(G, \chi, a, \a)$, Krop and Radford constructed an associated finite dimensional pointed rank one  Hopf algebra $H_{\mathcal{D}}$ and classified such Hopf algebras. They also described the Drinfeld doubles $D(H_{\mathcal D})$ of $H_{\mathcal D}$, see \cite{KR}.

Let $\mathcal{D}=(G, \chi, a, \a)$ be a group datum. The Hopf algebra $H_{\mathcal{D}}$ is generated as an algebra by $G$ and $x$ subject to the group relations for $G$, $$x^n=\a(a^n-1) \text{ and } xg=\chi(g) gx\ \text{ for all }g\in G.$$
The comultiplication $\t$ is given by $$\t(x)=x\ot a+1\ot x \text{ and } \t(g)=g\ot g, \ g\in G.$$
$H_{\mathcal{D}}$ has a $\Bbbk$-basis $\{gx^j|g\in G, 0\<j\<n-1\}$. If $\mathcal{D}$ is of {\it non-nilpotent type}, then $H_{\mathcal{D}}\cong H_{\mathcal{D'}}$  as Hopf algebras, where $\mathcal{D}'=(G, \chi, a, 1)$. Therefore, we always assume that $\a=1$ whenever $\mathcal{D}$ is of {\it non-nilpotent type}. For the details, one can refer to \cite{KR}.

In the sequel, we let $\mathcal{D}=(G, \chi, a, \a)$ be a group datum, and assume that $G$ is abelian. Let $\rho=\chi(a)$ and denote by $n$ the order of $\rho$. Let $\Gamma={\rm Hom}(G,\Bbbk^{\times})$ denote the group of $\Bbbk$-valued characters. Note that $D(H_{\mathcal D})=H_{\mathcal D}^{*\rm cop}\bowtie H_{\mathcal D}$ is generated as an algebra by its two sub-Hopf algebra $H_{\mathcal D}^{*\rm cop}$ and $H_{\mathcal D}$, where  $H^*_{\mathcal D}$ is the dual Hopf algebra of $H_{\mathcal D}$. $H_{\mathcal D}^{*\rm cop}$ is generated as an algebra by $\xi$ and $\Gamma$ subject to the following relations:
$$\xi^n=0 \text{ and }  \xi\gamma=\gamma(a)\gamma\xi \text{ for all } \gamma\in\Gamma.$$
The coalgebra structure of $H^{*\rm cop}_{\mathcal D}$ is determined by
$\t(\xi)=\xi\ot\varepsilon+ \chi\ot \xi$,
$\t(\g)=\g\ot \g$ if $\mathcal{D}$ is of nilpotent type, and
$$\t(\g)=\g\ot \g +(\g^n(a)-1)\sum_{l+r=n}\frac{1}{(l)_{\rho}!(r)_{\rho}!}\g\chi^l\xi^r\ot\g\xi^l$$
if $\mathcal{D}$ is of non-nilpotent type, where $\g\in\G$.

\begin{proposition}\cite[Proposition 5]{KR}\label{2.1}
The Drinfeld double $D(H_{\mathcal D})$ is generated as an algebra by $G$, $x$, $\Gamma$ and $\xi$ subject to the relations defining $H_{\mathcal D}$ and $H^{*\rm cop}_{\mathcal D}$ and the following relations:
\begin{enumerate}
\item[(a)] $g\g=\g g$ for all $g\in G$ and $\g\in \Gamma$;
\item[(b)] $\xi g=\chi^{-1}(g)g\xi$ for all $g\in G$;
\item[(c)] $[x,\xi]=a-\chi$;
\item[(d)] $\g(a)x\g=\g x$ if $\mathcal D$ is {\bf nilpotent};
\item[(e)] $\g(a)x\g=\g x+\frac{\g^n(a)-1}{(n-1)_\rho!}\g(\rho a-\chi)\xi^{n-1}$  if $\mathcal D$ is {\bf non-nilpotent}.
\end{enumerate}
\end{proposition}

\section{Simple modules and projective modules}\label{s3}

Throughout this and the next sections, for any $D(H_{\mathcal D})$-module $V$, let $P(V)$ and $I(V)$ denote the projective cover and the injective envelope of $V$, respectively. Let $l(V)$ and ${\rm rl}(V)$ denote the length and the radical length (Loewy length) of $V$, respectively. Moreover, let $sV$ denote the direct sum of $s$ copies of $V$ for any integer $s\>0$, where $sV=0$ when $s=0$. Let $V^x:=\{v\in V|xv=0\}$ and $V^{\xi}:=\{v\in V|\xi v=0\}$. For any $v_1, \cdots, v_s\in V$, denote by $\langle v_1, \cdots, v_s\rangle$ the submodule of $V$ generated by $\{v_1, \cdots, v_s\}$.

\subsection{Simple modules}\label{s3.1}

Krop and Radford classified the simple modules over $D(H_{\mathcal D})$ for any group datum $\mathcal{D}=(G,\chi,a,\a)$ with $G$ being abelian in \cite[Subsection 2.2]{KR}. In this subsection, we recall the simple $D(H_{\mathcal D})$-modules.

Let $\widehat{H^{*\rm cop}_{\mathcal D}}$ be the subalgebra of $D(H_{\mathcal D})$ generated by $H_{\mathcal D}^{*\rm cop}$ and $G$. Then $\widehat{H^{*\rm cop}_{\mathcal D}}=H^{*\rm cop}_{\mathcal D}\sharp \Bbbk G$, a smash product. One can easily check that $\widehat{H^{*\rm cop}_{\mathcal D}}\cong \Bbbk (G\times\Gamma)\oplus J(\widehat{H^{*\rm cop}_{\mathcal D}})$, where $J(\widehat{H^{*\rm cop}_{\mathcal D}})$ is the Jacobson radical of $\widehat{H^{*\rm cop}_{\mathcal D}}$. Thus, all simple  $\widehat{H^{*\rm cop}_{\mathcal D}}$-modules are one-dimensional. They are in 1-1 correspondence with characters of $G\times\Gamma$. For each character $\l: G\times \Gamma\rightarrow \Bbbk$, denote by $\Bbbk_{\l}$ the associated one dimensional $\widehat{H^{*\rm cop}_{\mathcal D}}$-module. Let $1_{\l}$ be a fixed non-zero element of $\Bbbk_{\l}$.
Then  $\Bbbk_{\l}$ is described by the relations $(g\g)1_{\l}=\l(g\g)1_{\l}$ and $J(\widehat{H^{*\rm cop}_{\mathcal D}})1_{\l}=0$, where we write $g\g$ for $(g,\g)\in G\times\Gamma$.

For a character $\l$, define a left $D(H_{\mathcal D})$-module $Z(\l):=D(H_{\mathcal D})\ot_{\widehat{H^{*\rm cop}_{\mathcal D}}}\Bbbk_{\l}$. For simplicity, the element $1\ot_{\widehat{H^{*\rm cop}_{\mathcal D}}}1_{\l}\in Z(\l)$ is still denoted by $1_{\l}$.
It is easy to see that $\{x^i1_{\l}|0\<i\<n-1\}$ is a $\Bbbk$-basis of $Z(\l)$. For an element $g\in G$, denote by $\hat{g}$ the character $\g\rightarrow \g(g)$ of $\Gamma$. Let $\phi$ be the character $\chi^{-1}\hat{a}$ of $G\times\Gamma$.

\begin{proposition}\cite[Proposition 6]{KR}\label{3.1}
 Let $Z(\l)$ be a $D(H_{\mathcal D})$-module. Then the following statements hold:
\begin{enumerate}
\item[(1)] $Z(\l)$ is simple if and only if $\l(a\chi^{-1})\notin\{1,\rho, \cdots, \rho^{n-2}\}$.
\item[(2)] If $\l(a\chi^{-1})=\rho^s$ for some $0\<s\<n-2$, then $Z(\l)$ contains a unique non-trivial submodule, which is generated by $x^{s+1}1_{\l}$.
\end{enumerate}
\end{proposition}

Let rad($Z(\l)$) denote the radical of $Z(\l)$. Then rad($Z(\l)$) is the unique proper submodule of $Z(\l)$ if $\l(a\chi^{-1})=\rho^s$ for some $0\<s\<n-2$, and rad($Z(\l)$)=0 otherwise. Set $L(\l):=Z(\l)/{\rm rad}Z(\l)$, and denote the element $1_{\l}+{\rm rad}Z(\l)\in L(\l)$ still by $1_{\l}$. Let $\Lambda:=\widehat{G\times \Gamma}$ be the set of characters of $G\times \Gamma$. Define ${\rm ev}_{a\chi^{-1}}: \L\rightarrow\Bbbk^{\times}$ by ${\rm ev}_{a\chi^{-1}}(\l)=\l(a\chi^{-1})$, $\l\in\L$, and let $K:={\rm Ker}({\rm ev}_{a\chi^{-1}})$.

\begin{proposition}\cite[Theorem 2]{KR}\label{3.2}
Retain the above notation. The following statements hold:
\begin{enumerate}
\item[(1)] $L(\mu)\cong L(\l)$ if and only if $\mu=\l$. There are $|G|^2$ pairwise non-isomorphic simple  $D(H_{\mathcal D})$-modules.
\item[(2)] For every $d$ with $1\<d\<n-1$, there are $|K|$ non-isomorphic simple $D(H_{\mathcal D})$-modules of dimension $d$.
\item[(3)] The number of $n$-dimensional simple $D(H_{\mathcal D})$-modules is $|G|^2-(n-1)|K|$ and is also greater than or equal to $|K|$.
\end{enumerate}
\end{proposition}

Let $\widetilde{\L}$ be the set of weights that map $a\chi^{-1}$ to a $\rho^s$ for some  $0\<s\<n-2$. For a $\l\in \L$, denote by $d(\l)$ the above integer $s$ if $\l\in \widetilde{\L}$, and set  $d(\l)=-1$ otherwise. Define a mapping $\sigma: \Lambda \rightarrow \Lambda$ by $\s(\l)=\l \phi^{d(\l)+1}$. Let $\tau=\sigma^2$.

\begin{lemma}\cite[Lemma 5]{KR}\label{3.3}
Retain the above  notation and let $\l\in\widetilde{\L}$. The following statements hold:
\begin{enumerate}
\item[(1)] $d(\s(\l))=n-d(\l)-2$.
\item[(2)] $\s(\widetilde{\L})=\widetilde{\L}$.
\item[(3)] Let $m=\frac{1}{n}{\rm ord}(a)$ if $\mathcal D$ is of non-nilpotent type, and $m=\frac{1}{n}{\rm lcm(ord}(a), {\rm ord}(\chi))$, otherwise. Then the mapping $\s$ has order $2m$, and $\tau$ has order $m$.
\item[(4)] $L(\tau^k\s(\l))\ncong L(\tau^{k'}(\l))$ for any $k,k'\in \mathbb{N}$.
\end{enumerate}
\end{lemma}

Throughout the following, let $m$ be the positive integer defined in Lemma \ref{3.3}(3).

\begin{remark}\label{3.4}
The character $\phi=\chi^{-1}\hat{a}$ has order $mn$. If $\mathcal D$ is of non-nilpotent type, then $m>1$.
In fact, if $\mathcal D$ is of non-nilpotent type and $m=1$, then ${\rm ord}(a)=n$ by Lemma \ref{3.3}(3), and hence $x^n=\a(a^n-1)=0$, a contradiction.
\end{remark}

Now we describe the simple modules $L(\l)$ more explicitely in terms of their dimensions. For any $1\<l\<n-1$, let
$$I_l=\{\l\in\widetilde{\L}|\l(a\chi^{-1})=\rho^{l-1}\}, \   I_n:=\L\backslash\widetilde{\L}=\{\l\in\L|\l\notin\widetilde{\L}\}, \mathrm{and}$$
$$I'_{n}=\{\l\in I_n|\l(a\chi^{-1})\neq \rho^k,  \ 0\<k\<n-1\}, \ I''_{n}=\{\l\in I_n|\l(a\chi^{-1})=\rho^{n-1}\}.$$ Then $I_n=I'_n\cup I''_n$, $\cup_{l=1}^{n-1} I_l=\widetilde{\L}$ and $\cup_{l=1}^{n} I_l={\L}$.
For any $1\<l\<n$ and $\l\in I_l$, denote by $V(l,\l)$ the $l$-dimensional simple $D(H_{\mathcal D})$-module $L(\l)$.

\begin{corollary}\label{3.5}
Let $1\<l\<n-1$ and $\l\in I_l$. Then $\s(\l), \s^{-1}(\l)\in I_{n-l}$. Moreover,
$\s^{-1}(\l)=\l\phi^{d(\l)+1-n}=\l\phi^{l-n}$.
\end{corollary}
\begin{proof}
It follows from Lemma \ref{3.3} and a straightforward verification.
\end{proof}

In the sequel, we describe  $\Bbbk$-bases of simple modules $V(l,\l)$.  For the convenience in classifying the non-simple indecomposable $D(H_{\mathcal D})$, we will need two types of bases for each simple $D(H_{\mathcal D})$-module $V(l,\l)$.
For any $1\<l\<n$ and $\l\in I_l$, define
$$\a_i(l,\l):=(i)_{\rho}(\l(\chi)-\l(a)\rho^{1-i})\in\Bbbk, \ i\>1.$$
Since $\s^2(\l)=\l\phi^n$, we have $\a_i(l,\tau(\l))=\a_i(l,\l)$.  We define $\b(l, \l)\in\Bbbk$ by
$\b(1, \l):=1$ if $l=1$ and $\b(l, \l):=\a_1(l,\l)\a_2(l,\l)\cdots\a_{l-1}(l,\l)$ if $1<l\<n$.
Note that $\a_i(l,\l)\neq 0$ for all $1\<i\<l-1$, and hence $\b(l,\l)\neq 0$.

Let $1\<l\<n$ and $\l\in I_l$. Define $v_i, m_i\in V(l,\l)$, $0\<i\<l-1$, as follows:
$$v_i:=x^i1_{\l};   \  \  m_i:=\left\{\begin{array}{ll}
\a_{i+1}(l,\l)\a_{i+2}(l,\l)\cdots\a_{l-1}(l,\l)x^i1_{\l},& \text{ if } 0\<i\<l-2,\\
x^{l-1}1_{\l},& \text{ if } i=l-1.\\
\end{array}\right.$$
Then both $\{v_i|0\<i\<l-1\}$ and $\{m_i|0\<i\<l-1\}$ form  $\Bbbk$-bases of $V(l,\l)$. The former is called the {\it natural basis} of $V(l,\l)$ and the latter is called the {\it standard basis} of $V(l,\l)$.

 We now present the $D(H_{\mathcal D})$-module actions on both bases.
For $1\<l\<n-1$,  the $D(H_{\mathcal D})$-module action on the natural basis of $V(l,\l)$ is determined by
\begin{equation*}
\begin{array}{ll}
\vspace{0.2cm}
(g\g)v_i=(\phi^i\l)(g\g)v_i, & g\in G, \ \g\in\G, \ 0\<i\<l-1,\\
xv_i=\left\{
\begin{array}{ll}
v_{i+1},& 0\<i\<l-2,\\
0,& i=l-1,\\
\end{array}\right.&
\xi v_i=\left\{
\begin{array}{ll}
0, & i=0,\\
\a_i(l,\l)v_{i-1},& 1\<i\<l-1, \\
\end{array}\right.\\
\end{array}
\end{equation*}
while the $D(H_{\mathcal D})$-module action on the standard basis of  $V(l,\l)$ is  given by
\begin{equation*}
\begin{array}{ll}
\vspace{0.2cm}
(g\g)m_i=(\phi^i\l)(g\g)m_i, & g\in G, \ \g\in\G, \ 0\<i\<l-1,\\
xm_i=\left\{
\begin{array}{ll}
\a_{i+1}(l,\l)m_{i+1},& 0\<i\<l-2,\\
0,& i=l-1,\\
\end{array}\right.&
\xi m_i=\left\{
\begin{array}{ll}
0, & i=0,\\
m_{i-1},& 1\<i\<l-1.\\
\end{array}\right.\\
\end{array}
\end{equation*}
It is clear that  $V(l,\l)^x=\Bbbk v_{l-1}=\Bbbk m_{l-1}$ and $V(l,\l)^{\xi}=\Bbbk v_0=\Bbbk m_0$.

Next we consider the case that $l=n$ with $\l\in I_n$. In this situation, $V(n,\l)$ is projective. To describe the $D(H_{\mathcal D})$-module structure on  $V(n,\l)$,  it is necessary to distinguish between two cases for  $\mathcal{D}$:   the nilpotent case and the non-nilpotent case.

If $\mathcal{D}$ is of {\bf nilpotent type},  the $D(H_{\mathcal D})$-module action on the natural basis of  $V(n,\l)$ is given by
\begin{equation*}
	\begin{array}{ll}
		\vspace{0.2cm}
		(g\g)v_i=(\phi^i\l)(g\g)v_i, & g\in G, \ \g\in\G, \ 0\<i\<n-1,\\
		xv_i=\left\{
		\begin{array}{ll}
			v_{i+1},& 0\<i\<n-2,\\
			0,& i=n-1,\\
		\end{array}\right.&
		\xi v_i=\left\{
		\begin{array}{ll}
			0, & i=0,\\
			\a_i(l,\l)v_{i-1},& 1\<i\<n-1, \\
		\end{array}\right.\\
	\end{array}
\end{equation*}
while the $D(H_{\mathcal D})$-module action on the standard basis of $V(n, \l)$ is  given by
\begin{equation*}
\begin{array}{ll}
\vspace{0.2cm}
(g\g)m_i=(\phi^i\l)(g\g)m_i, &g\in G, \ \g\in\G, \ 0\<i\<n-1,\\
xm_i=\left\{
\begin{array}{ll}
\a_{i+1}(n,\l)m_{i+1},& 0\<i\<n-2,\\
0,& i=n-1,\\
\end{array}\right.&
\xi m_i=\left\{
\begin{array}{ll}
0, & i=0,\\
m_{i-1},& 1\<i\<n-1.\\
\end{array}\right.\\
\end{array}
\end{equation*}
It is clear that  $V(n,\l)^x=\Bbbk v_{n-1}=\Bbbk m_{n-1}$ and $V(n,\l)^{\xi}=\Bbbk v_0=\Bbbk m_0$.

If  $\mathcal{D}$ is of {\bf non-nilpotent type}, then the $D(H_{\mathcal D})$-module actions on the two bases of $V(n,\l)$ are given  respectively by
\begin{equation*}
	\begin{array}{ll}
		\vspace{0.2cm}
		(g\g)v_i=(\phi^i\l)(g\g)v_i, & g\in G, \ \g\in\G, \ 0\<i\<n-1,\\
		xv_i=\left\{
		\begin{array}{ll}
			v_{i+1},& 0\<i\<n-2,\\
			(\l^n(a)-1)v_0,& i=n-1,\\
		\end{array}\right.&
		\xi v_i=\left\{
		\begin{array}{ll}
			0, & i=0,\\
			\a_i(l,\l)v_{i-1},& 1\<i\<n-1.\\
		\end{array}\right.\\
	\end{array}
\end{equation*}
and
\begin{equation*}
\begin{array}{ll}
\vspace{0.2cm}
(g\g)m_i=(\phi^i\l)(g\g)m_i, &g\in G, \ \g\in\G, \ 0\<i\<n-1,\\
xm_i=\left\{
\begin{array}{ll}
\a_{i+1}(n,\l)m_{i+1},& 0\<k\<n-2,\\
\frac{\l^n(a)-1}{\b(n, \l)}m_0,& i=n-1,\\
\end{array}\right.&
\xi m_i=\left\{
\begin{array}{ll}
0, & i=0,\\
m_{i-1},& 1\<i\<n-1.\\
\end{array}\right.\\
\end{array}
\end{equation*}
Clearly, $V(n,\l)^{\xi}=\Bbbk v_0=\Bbbk m_0$.

By the above discussion, we have following explicit description of  the simple $D(H_{\mathcal D})$-modules.

\begin{proposition}\label{3.6}
The following set $$\{V(l,\l)|1\<l\<n,\l\in I_l\}$$ gives a complete set of  representatives of isomorphism classes of simple $D(H_{\mathcal D})$-modules. Moreover, $V(n,\l)$ is a projective $D(H_{\mathcal D})$-module for any $\l \in  I_n$.
\end{proposition}

\begin{corollary}\label{3.7}
Let $M$ be a finite dimensional semisimple $D(H_{\mathcal D})$-module. Then $l(M)={\rm dim}(M^{\xi})$. Furthermore, if $\mathcal D$ is of nilpotent type, then $l(M)={\rm dim}(M^x)$.
\end{corollary}

\subsection{Indecomposable projective modules}\label{s3.2}

Krop and Radford described all projective indecomposable $D(H_{\mathcal D})$-modules for any group datum $\mathcal{D}=(G,\chi,a,\mu)$ with $G$ abelian in \cite[Subsection 2.3]{KR}, including their radical series and composition factors.

Since $D(H_{\mathcal D})$ is a symmetric algebra, we have $P(V)\cong I(V)$ and
$$P(V)/{\rm rad}(P(V))\cong{\rm soc}(P(V))\cong V$$
 for any simple $D(H_{\mathcal D})$-module $V$.
Let $J$ denote the Jacobson radical of $D(H_{\mathcal D})$, and let $P(\l)$ be the projective cover of $L(\l)$. Denote by $J^i(\l)$, $i=0,1,\cdots,$ the terms of radical series of $P(\l)$. Then, by \cite[Subsection 2.3]{KR}, we have the following proposition and corollary.

\begin{proposition}\cite[Theorem 3]{KR}\label{3.8}
Let $\l\in \widetilde{\Lambda}$. The radical series of $P(\l)$ is given by
$$P(\l)\supset J(\l)\supset L(\l)\supset 0$$
with  $$J(\l)/L(\l)\cong L(\sigma(\l))\oplus  L(\sigma^{-1}(\l)).$$
\end{proposition}

\begin{corollary}\label{3.9}
Let $\l\in\L$. Then:
\begin{enumerate}
\item[(1)] If $\l\in I_n$, the simple module $Z(\l)$ is projective.
\item[(2)] If $\l\in\widetilde{\L}$, the radical length of $P(\l)$ is  ${\rm rl}(P(\l))=3$ and its length is $l(P(\l))=4$.
\item[(3)]   The radical length of $D(H_{\mathcal D})$ is   ${\rm rl}(D(H_{\mathcal D}))=3$; in particular,  $J^3=0$.
\end{enumerate}
\end{corollary}
For a more explicit description of the structure of indecomposable projective modules, we consider separately the two cases in which $\mathcal{D}$ is nilpotent or non-nilpotent.

{\bf Case 1}: $\mathcal{D}$ is of {\bf nilpotent type}.

Let $1\<l\<n-1$, $\l \in I_l$, and let $P(l,\l)$ be a $2n$-dimensional vector space with a $\Bbbk$-basis $\{v_0, v_1, \cdots, v_{n-1}, u_0, u_1, \cdots, u_{n-1}\}$.  A straightforward verification shows that $P(l,\l)$ becomes a $D(H_{\mathcal D})$-module under the action defined by
\begin{equation*}
\begin{array}{ll}
\vspace{0.2 cm}
(g\g)v_i=(\phi^i\l)(g\g)v_i, & g\in G, \ \g\in\G, \ 0\<i\<n-1,\\
\vspace{0.2 cm}
(g\g)u_i=(\phi^{i-n+l}\l)(g\g)u_i, & g\in G, \ \g\in\G, \ 0\<i\<n-1,\\
\vspace{0.2 cm}
xv_i=\left\{
\begin{array}{ll}
v_{i+1}, & 0\<i\<n-2,\\
0,& i=n-1,\\
\end{array}\right.&
x u_i=\left\{
\begin{array}{ll}
u_{i+1}, & 0\<i\<n-2,\\
0,& i=n-1,\\
\end{array}\right.\\
\end{array}
\end{equation*}
\begin{equation*}
\begin{array}{ll}
\vspace{0.2 cm}
\xi v_i=\left\{
\begin{array}{ll}
u_{n-l-1},& i=0,\\
\a_i(l,\l)v_{i-1}+u_{n-l+i-1},& 1\<i\<l-1,\\
u_{n-1},& i=l,\\
\a_{i-l}(n-l,\s(\l))v_{i-1},& l+1\<i\<n-1,\\
\end{array}\right.&\\
\vspace{0.2 cm}
\xi u_i=\left\{
\begin{array}{ll}
0,& i=0,\\
\a_i(n-l,\s^{-1}(\l))u_{i-1}, & 1\<i\<n-l-1,\\
0,& i=n-l,\\
\a_{i-n+l}(l,\l)u_{i-1}, & n-l+1\<i\<n-1.\\
\end{array}\right.\\
\end{array}
\end{equation*}
The above basis is called the {\it standard basis of $P(l,\l)$}. Clearly, $P(l,\l)^{x}=\Bbbk v_{n-1}+\Bbbk u_{n-1}$ and $P(l,\l)^{\xi}=\Bbbk u_0+\Bbbk u_{n-l}$.

{\bf Case 2}: $\mathcal{D}$ is of {\bf non-nilpotent type}.
Let $1\<l\<n-1$, and  $\l \in I_l$, and consider a $2n$-dimensional vector space  $P(l,\l)$  with a $\Bbbk$-basis $\{v_0, v_1, \cdots, v_{n-1}, u_0, u_1, \cdots, u_{n-1}\}$. It can be easily checked that  $P(l, \l)$ carries a structure of a $D(H_{\mathcal D})$-module with the action defined  by
\begin{equation*}
\begin{array}{ll}
\vspace{0.2 cm}
(g\g)v_i=(\phi^{i-n+l}\l)(g\g)v_i, & g\in G, \ \g\in\G, \ 0\<i\<n-1,\\
(g\g)u_i=(\phi^i\l)(g\g)u_i, & g\in G, \ \g\in\G, \ 0\<i\<n-1,\\
\end{array}
\end{equation*}
\begin{equation*}
\begin{array}{ll}
\vspace{0.2 cm}
x v_i=\left\{
\begin{array}{ll}
\a_{i+1}(n-l,\s^{-1}(\l))v_{i+1},& 0\<i\<n-l-2,\\
u_0,& i=n-l-1,\\
\a_{i+1-n+l}(l,\l)v_{i+1}+u_{i+1-n+l},& n-l\<i\<n-2,\\
y_{l,\l}v_0+u_l,& i=n-1,\\
\end{array}\right.&\\
x u_i=\left\{
\begin{array}{ll}
\a_{i+1}(l,\l)u_{i+1},& 0\<i\<l-2,\\
0,& i=l-1,\\
\a_{i+1-l}(n-l,\s(\l))u_{i+1}, & l\<i\<n-2,\\
z_{l,\l}u_0,& i=n-1,\\
\end{array}\right.\\
\end{array}
\end{equation*}
\begin{equation*}
\begin{array}{ll}
\xi v_i=\left\{
\begin{array}{ll}
0, & i=0,\\
v_{i-1},& 1\<i\<n-1,\\
\end{array}\right.&
\xi u_i=\left\{
\begin{array}{ll}
0, & i=0,\\
u_{i-1},& 1\<i\<n-1,\\
\end{array}\right.\\
\end{array}
\end{equation*}
where $y_{l,\l}=\frac{\rho^{1-l}\l(a)-\rho^l\l(\chi)}{(n-1)!_{\rho}}$ and $z_{l,\l}=\frac{\rho\l(a)-\l(\chi)}{(n-1)!_{\rho}}$. Moreover, $y_{l,\l}+z_{l,\l}=0$ and $z_{n-l,\s^{-1}(\l)}=y_{l,\l}$ by $\l(\chi)=\l(a)\rho^{1-l}$.

Such a basis is referred to as  the {\it  standard basis} of  $P(l,\l)$. Clearly, $P(l,\l)^x=\Bbbk u_{l-1}+\Bbbk(u_{n-1}-z_{l,\l}v_{n-l-1})$ and $P(l,\l)^{\xi}=\Bbbk v_0+\Bbbk u_0$.

\begin{proposition}\label{3.10}
Let $1\<l\<n-1$ and $\l \in I_l$. Then $P(l,\l)\cong P(\l)$.
\end{proposition}
\begin{proof}
Let $\{v_0, v_1, \cdots, v_{n-1}, u_0, u_1, \cdots, u_{n-1}\}$ be the standard basis of $P(l,\l)$ as stated above.

{\bf Case 1}: $\mathcal{D}$ is of nilpotent type. Let $M:=\langle u_{n-1}\rangle$.
Then it is easy to verify that $M={\rm span}\{u_{n-l},u_{n-l+1},\cdots,u_{n-1}\}$. Using the natural basis of $V(l,\l)$, one  sees that $M\cong V(l,\l)$.
Let $N$ be a simple submodule of $P(l,\l)$.  By Corollary \ref{3.7}, we have ${\rm dim}(N^x)=1$. Choose  $0\neq z\in N^x$.  Since $N^x\subseteq P(l,\l)^x$, we can write
$$z=\b_1v_{n-1}+\b_2u_{n-1}, \ \mathrm{for\, some}\  \b_1, \b_2\in\Bbbk.$$
If $\b_1\neq 0$, then
$$az-\rho^{1-l}\l(a)z=(\rho-\rho^{1-l})\l(a)\b_1v_{n-1}.$$
 Hence $v_{n-1}\in N$,  and  therefore  $\xi^{n-l}v_{n-1}=\b(n-l, \s(\l))u_{n-1}\in N$. Thus, $v_{n-1}, u_{n-1}\in N$, implying ${\rm dim}(N^x)=2$, a contradiction. Therefore,  $\b_1=0$ and $\b_2\neq 0$, so  $u_{n-1}=\b_2^{-1}z\in N$. This shows  $M=\langle u_{n-1}\rangle\subseteq N$,  and since $M$ and $N$ are both simple, we have $N=M$.
 Thus, ${\rm soc}P(l,\l)=M\cong V(l,\l)$. Since $D(H_{\mathcal D})$ is a symmetric algebra, it follows that  $P(l,\l)$ is isomorphic to a submodule of $P(\l)$.

Furthermore,  ${\rm dim}((P(l,\l)/{\rm soc}P(l,\l))^x)=2$, and  thus,  by Corollary \ref{3.7}, \\
$l({\rm soc}(P(l,\l)/{\rm soc}P(l,\l)))\<2$.
Let $\ol{v}$ denote the image of $v\in P(l,\l)$ under the canonical epimorphism $P(l,\l)\ra P(l,\l)/{\rm soc}P(l,\l)$. Then one can check that
$$M_1:={\rm span}\{\ol{v_l},\ol{v_{l+1}},\cdots,\ol{v_{n-1}}\}, \ M_2:={\rm span}\{\ol{u_0},\ol{u_1},\cdots,\ol{u_{n-l-1}}\}$$
 are submodules of $P(l,\l)/{\rm soc}P(l,\l)$. By Corollary \ref{3.5}, we have
 $$M_1\cong V(n-l,\s(\l)),  \mathrm{and}\  M_2\cong V(n-l,\s^{-1}(\l)).$$
  It follows that  ${\rm soc}(P(l,\l)/\mathrm{soc}P(l,\l))=M_1\oplus M_2$, and consequently,
$${\rm soc}^2P(l,\l)={\rm span}\{v_l, \cdots, v_{n-1}, u_0,\cdots,u_{n-1}\}.$$
Now it is straightforward to see that
$$P(l,\l)/{\rm soc}^2P(l,\l)\cong V(l,\l).$$
Therefore, the length of $P(l,\l)$ is  $l(P(l,\l))=4$, and  by Corollary \ref{3.9}(2), $P(l,\l)\cong P(\l)$.

{\bf Case 2}: $\mathcal{D}$ is of non-nilpotent type.  Let
$M:=\langle u_0\rangle$. It is straightforward to verify that $M={\rm span}\{u_0, u_1, \cdots, u_{l-1}\}$. Using the standard basis of $V(l,\l)$, one sees that  $M\cong V(l,\l)$. \\
Let $N$ be a simple submodule of $P(l,\l)$. By Corollary \ref{3.7}, we have  ${\rm dim}(N^{\xi})=1$. Choose $0\neq z\in N^{\xi}$.  Since $N^{\xi}\subseteq P(l,\l)^{\xi}$, we can write $z=\b_1v_0+\b_2u_0$ , for some $\b_1, \b_2\in\Bbbk$. If $\b_1\neq 0$, then
$$az-\l(a)z=(\rho^{n-l}-1)\l(a)\b_1v_0.$$
 Hence $v_0\in N$,  and  therefore  $x^{n-l}v_0=\b(n-l,\s^{-1}(\l))u_0\in N$. Thus, $v_0, u_0\in N$, which implies  ${\rm dim}(N^{\xi})=2$, a contradiction. Therefore,  $\b_1=0$ and $\b_2\neq 0$, and so $u_0=\b_2^{-1}z\in N$. This shows $M=\langle u_0\rangle\subseteq N$,  and  since both $M$ and $N$ are simple, we conclude $N=M$.
Hence, ${\rm soc}P(l,\l)=M\cong V(l,\l)$. Because  $D(H_{\mathcal D})$ is a symmetric algebra, it follows that  $P(l,\l)$ is isomorphic to a submodule of $P(\l)$.
Clearly, $${\rm dim}((P(l,\l)/{\rm soc}P(l,\l))^{\xi})=2.$$
Therefore,  by Corollary \ref{3.7}, we  have
$$l({\rm soc}(P(l,\l)/{\rm soc}P(l,\l)))\<2.$$
Let $\ol{v}$ denote the image of $v\in P(l,\l)$ under the canonical epimorphism $P(l,\l)\ra P(l,\l)/{\rm soc}P(l,\l)$. Then one can  verify  that
$$M_1:={\rm span}\{\ol{v_0},\ol{v_1},\cdots,\ol{v_{n-l-1}}\}, \  M_2:={\rm span}\{\ol{u_l},\ol{u_{l+1}},\cdots,\ol{u_{n-1}}\}$$
 are submodules of $P(l,\l)/{\rm soc}P(l,\l)$,  satisfying
 $$M_1\cong V(n-l,\s^{-1}(\l))\ \mathrm{and}\  M_2\cong V(n-l,\s(\l)).$$
It follows that
$${\rm soc}(P(l,\l)/{\rm soc}P(l,\l))=M_1\oplus M_2,$$
 and consequently,
$${\rm soc}^2P(l,\l)={\rm span}\{v_0, v_1, \cdots, v_{n-l-1}, u_0, u_1, \cdots, u_{n-1}\}.$$
Now one can easily check that $P(l,\l)/{\rm soc}^2P(l,\l)\cong V(l,\l)$.
Therefore, the length of $P(l,\l)$ is $l(P(l,\l))=4$, and by Corollary \ref{3.9}(2), $P(l,\l)\cong P(\l)$.
\end{proof}

\begin{lemma}\label{3.11}
Let $1\<l\<n-1$ and $\l\in I_l$. Then:
\begin{enumerate}
\item[(1)] {\rm soc}$P(l,\l)={\rm rad}^2P(l,\l)\cong V(l,\l)$ and ${\rm soc}^2 P(l,\l)={\rm rad} P(l,\l)$.
\item[(2)] ${\rm soc}^2 P(l,\l)/{\rm soc} P(l,\l)=V(n-l,\s(\l)) \oplus V(n-l,\s^{-1}(\l))$.
\end{enumerate}
\end{lemma}
\begin{proof}
It follows from Proposition \ref{3.10} and its proof.
\end{proof}

\begin{corollary}\label{3.12}
A complete set of representatives for the  isomorphism classes of indecomposable projective $D(H_{\mathcal D})$-modules is given by
 $$\{P(l,\l),V(n,\mu)|1\<l\<n-1, \l\in I_l, \mu \in I_n \}.$$
\end{corollary}
\begin{proof}
It follows from Propositions \ref{3.6} and  \ref{3.10}.
\end{proof}

\begin{corollary}\label{3.13}
If $P$ is a non-simple indecomposable projective $D(H_{\mathcal D})$-module, then ${\rm rad}P={\rm soc}^2P$ and ${\rm rad}^2 P={\rm soc}P$.
\end{corollary}
\begin{proof}
It follows from Proposition \ref{3.6}, Lemma \ref{3.11} and Corollary \ref{3.12}.
\end{proof}

\begin{lemma}\label{3.14}
Let $M$ be an indecomposable $D(H_{\mathcal D})$-module. If ${\rm rl}(M)$={\rm 1}, then $M\cong V(l,\l)$ for some $1\<l\<n$ and $\l \in I_l$. If ${\rm rl}(M)=3$, then $M\cong P(l,\l)$ for some $1\<l\<n-1$ and $\l \in I_l$.
\end{lemma}
\begin{proof}
It follows from  Proposition \ref{3.6}, Corollary \ref{3.12} and \cite[Lemma 3.5]{Ch3}.
\end{proof}

\begin{proposition}\cite[Theorem 4]{KR}\label{3.15}
Each block of $D(H_{\mathcal D})$ is either isomorphic to $M_{n}(\Bbbk)$, or isomorphic as a left $D(H_{\mathcal D})$-module to $$(\oplus_{k=0}^{m-1}lP(l,\tau^k(\l)))\oplus(\oplus_{k=0}^{m-1}(n-l)P(n-l,\tau^k(\s(\l)))), $$
where $1\<l\<n-1$ and $\l \in I_l$.
\end{proposition}

\section{Indecomposable modules with Loewy length two}\label{s4}

In this section, we investigate the non-simple, non-projective indecomposable $D(H_{\mathcal D})$-modules. By Lemma \ref{3.14}, such indecomposable modules have Loewy length 2.

\begin{proposition}\label{4.1}
Let $M$ be an indecomposable $D(H_{\mathcal D})$-module with ${\rm rl}(M)=2$. Then there exist a $1\<l\<n-1$ and a $\l \in I_l$ such that $I(M)\cong\oplus_{k=0}^{m-1}s_kP(l,\tau^k(\l))$ for some $s_k\in \mathbb{N}$, $0\<k\<m-1$.
\end{proposition}
\begin{proof}
Clearly, $M$ is an indecomposable module over a non-simple block of $D(H_{\mathcal D})$, and its injective hull $I(M)$ is projective. By Proposition \ref{3.15}, we can write $I(M)=P_1\oplus P_2$, where
$$P_1\cong\oplus_{k=0}^{m-1}s_kP(l,\tau^k(\l))\ \mathrm{and}\  P_2\cong\oplus_{k=0}^{m-1}t_kP(n-l,\tau^k(\s(\l)))$$
 for some character $\l\in I_l$, $1\<l\<n-1$ and integers $s_k, t_k\in\mathbb N$, $0\<k\<m-1$.
 We may view $M$ as a submodule of $I(M)$ and set $M_1=M\cap P_1$ and $M_2=M\cap P_2$. Then ${\rm soc}M={\rm soc}P_1\oplus{\rm soc}P_2$ and so ${\rm soc}M\subset M_1\oplus M_2$. Since rl$(M)=2$, the quotient $M/{\rm soc}M$ is semisimple. Hence
 $$M/{\rm soc}M\subseteq{\rm soc}^2I(M)/{\rm soc}I(M),$$
  and  consequently  $M\subseteq{\rm rad}I(M)={\rm soc}^2I(M)$ by Corollary \ref{3.13}.

Let $\pi:I(M)\rightarrow I(M)/{\rm soc} I(M)$ be the canonical epimorphism. Then
$$M/{\rm soc}M\subseteq{\rm rad}I(M)/{\rm rad}^2I(M)=\pi({\rm rad}P_1)\oplus\pi({\rm rad}P_2).$$
By Lemma \ref{3.11}, we have
$$\begin{array}{rcl}
\pi({\rm rad}P_1)&\cong&{\rm rad}P_1/{\rm rad}^2P_1\cong\oplus_{k=0}^{m-1}(s_k+s_{k+1})V(n-l,\tau^k(\s(\l))),\\
\pi({\rm rad}P_2)&\cong&{\rm rad}P_2/{\rm rad}^2P_2\cong\oplus_{k=0}^{m-1}(t_{k-1}+t_k)V(l,\tau^k(\l)),\\
\end{array}$$
where $s_m=s_0$ and $t_{-1}=t_{m-1}$. By Lemma \ref{3.3}(4), if $V$ is a simple submodule of $M/{\rm soc}M$, then  either $V\subseteq\pi({\rm rad}P_1)$ or $V\subseteq\pi({\rm rad}P_2)$. In the former case,
$$\pi^{-1}(V)\subseteq M\cap({\rm rad}P_1+{\rm rad}^2I(M))=M\cap({\rm rad}P_1\oplus{\rm rad}^2P_2)\subseteq M_1\oplus M_2.$$
Similarly, if $V\subseteq\pi({\rm rad}P_2)$,  then $\pi^{-1}(V)\subseteq M_1\oplus M_2$.
Therefore, $M\subseteq M_1\oplus M_2$, and hence $M=M_1\oplus M_2$.

Since $M$ is indecomposable, it follows that $M_1=0$ or $M_2=0$, and consequently $P_1=0$ or $P_2=0$.  Thus, we conclude either
 $$I(M)\cong\oplus_{k=0}^{m-1}s_kP(l,\tau^k(\l)),\  \mathrm{or}\  I(M)\cong\oplus_{k=0}^{m-1}t_kP(n-l,\tau^k(\s(\l))).$$
\end{proof}

Let $M$ be an indecomposable $D(H_{\mathcal D})$-module with ${\rm rl}(M)=2$. By \cite[Lemma 3.7]{Ch3}, we have
${\rm soc}M={\rm rad}M$. If $l({\rm soc}M)=t$ and $l(\text{head}M)=s$, then we say that $M$ is of $(s,t)$-type (cf. \cite{Ch3}).

\begin{corollary}\label{4.2}
Let $M$ be of $(s,t)$-type such that
$${\rm soc}M\cong\oplus_{k=0}^{m-1}t_kV(l,\tau^k(\l))$$
 for some $1\<l\<n-1$ and $\l \in I_l$.  Then:
\begin{enumerate}
\item[(1)]  ${\rm head}M\cong\oplus_{k=0}^{m-1}s_kV(n-l,\tau^k(\s(\l)))$ for some integers $s_k\>0$.
\item[(2)] ${\rm Hom}_{D(H_{\mathcal D})}(M, {\rm soc}M)={\rm Hom}_{D(H_{\mathcal D})}({\rm head}M, {\rm soc}M)=0$.
\end{enumerate}
\end{corollary}
\begin{proof}
(1) follows from the proof of Proposition \ref{4.1}. Since ${\rm rad}M={\rm soc}M$ and ${\rm soc}M$ is semisimple, it follows from (1) and together with Lemma \ref{3.3}(4) that
$$\begin{array}{rl}
&{\rm Hom}_{D(H_{\mathcal D})}(M, {\rm soc}M)={\rm Hom}_{D(H_{\mathcal D})}({\rm head}M, {\rm soc}M)\\
\cong&{\rm Hom}_{D(H_{\mathcal D})}(\oplus_{k=0}^{m-1}s_kV(n-l,\tau^k(\s(\l))), \oplus_{k=0}^{m-1}t_kV(l,\tau^k(\l)))=0.\\
\end{array}$$
This proves (2).
\end{proof}

%{\bf The modules $\Omega^{\pm s}V(l,\l)$ with $s\>1$, $1\<l\<n-1$, $\l\in I_l$:}
Let $M$ be an $D(H_{\mathcal D})$-module. Fix a projective cover $f: P(M)\ra M$ and define the syzygy $\Omega M$ of $M$ to be ${\rm Ker}f$. Dually, fix an injective envelope $f: M\ra I(M)$ and define the cosyzygy $\Omega^{-1}M$ of $M$ to be ${\rm Coker}f$.
%If $M$ has no nonzero projective (injective) direct summands, then neither do $\Omega M$ and $\Omega^{-1}M$, and $\Omega\Omega^{-1}M\cong M\cong\Omega^{-1}\Omega M$. Moreover, $M$ is indecomposable if and only if $\Omega M$ is indecomposable. If $N$ is also an $D(H_{\mathcal D})$-module without nonzero projective (injective) direct summands, then $M\cong N$ if and only if $\Omega M\cong\Omega N$, cf. \cite[p.126]{ARS}.
By \cite[p.126]{ARS}, one obtains a family of indecomposable $D(H_{\mathcal D})$-modules $$\Omega^{\pm s}V(l,\l),\  \mathrm{where}\  s\>1, \ 1\<l\<n-1, \l\in I_l.$$

%\begin{lemma}\label{4.2}
%Let $M$ be an indecomposable $D(H_{\mathcal D})$-module with ${\rm rl}(M)=2$.
%\begin{enumerate}
%\item[(1)] $M$ is of $(1,2)$-type $\Leftrightarrow$ $M\cong \Omega^{-1} V(l,\l)$ for some $1\<l\<n-1$ and $\l \in I_l$.
%\item[(2)] $M$ is of $(2,1)$-type $\Leftrightarrow$ $M\cong \Omega V(l,\l)$  for some $1\<l\<n-1$ and $\l \in I_l$.
%\end{enumerate}
%\end{lemma}
%\begin{proof}
%(1) By Proposition \ref{3.10} and its proof, if $M\cong\Omega^{-1} V(l,\l)$ for some $1\<l\<n-1$ and $\l\in I_l$, then $M$ is of $(1,2)$-type. Conversely, let $M$ is of $(1,2)$-type. Then
%$M/{\rm rad}M=M/{\rm soc}M\cong V(l,\l)$ for  some $1\<l\<n-1$ and $\l\in I_l$.
% Then there is a $D(H_{\mathcal D})$-module epimorphism $f: P(l,\l)\ra M$. Since $l(M)=3$ and $l(P(l, \l))=4$, $l({\rm Ker}f)=1$. Hence ${\rm Ker}f={\rm soc}P(l,\l)\cong V(l,\l)$, and so $M\cong P(l, \l)/{\rm Ker}f=P(l,\l)/{\rm soc}P(l,\l)\cong\Omega^{-1}V(l,\l)$.

%(2) It is similar to (1) or dual to (1).
%\end{proof}

\begin{lemma}\label{4.3}
Let $M$ be of $(s,t)$-type.
\begin{enumerate}
\item[(1)] $t\<2s$, moreover, if $s\neq 1$ then $t<2s$ and $\Omega M$ is of $(2s-t,s)$-type.
\item[(2)] $s\<2t$, moreover, if $t\neq 1$ then $s<2t$ and $\Omega^{-1} M$ is of $(t,2t-s)$-type.
\item[(3)] If $s=t$, then both $\Omega M$ and $\Omega^{-1} M$ are of $(t,t)$-type.
\end{enumerate}
\end{lemma}
\begin{proof}
(1) By the proof of Proposition \ref{4.1}, we have
$${\rm head}M\cong\oplus_{k=0}^{m-1}s_kV(l,\tau^k(\l))\ \mathrm{and}\ P(M)\cong\oplus_{k=0}^{m-1}s_kP(l,\tau^k(\l))$$
for some $s_k\in \mathbb{N}$, $0\<k\<m-1$ with $\sum_{k=0}^{m-1}s_k=s$.
Let $f: P(M)\rightarrow M$ be an epimorphism. Then the restriction $f|_{{\rm rad}P(M)}: {\rm rad}P(M)\rightarrow{\rm rad}M$ remains an epimorphism. Since ${\rm rad}M={\rm soc}M$ is semisimple, we have ${\rm rad}^2P(M)\subseteq{\rm Ker}(f|_{{\rm rad}P(M)})$. Hence,
$$l({\rm soc}M)\<l({\rm rad}P(M)/{\rm rad}^2P(M))=2s,$$
 and  therefore $t\<2s$.

Now assume $s\neq1$. Then ${\rm soc}P(M)={\rm rad}^2P(M)\subseteq{\rm Ker}f$. Thus,  $f$ induces an epimorphism
$$\ol{f}: P(M)/{\rm rad}^2P(M)\rightarrow M. $$
 If $t=2s$, then $l(M)=3s=l(P(M)/\text{soc}P(M))$,  so  $\ol{f}$ would be  an isomorphism. However, $P(M)/\text{rad}^2P(M)$ is not indecomposable when $s>1$, a contradiction. Therefore, $t<2s$. Meanwhile, since $\Omega M\cong {\rm Ker}f\subseteq{\rm rad}P(M)$ and ${\rm rad}^2P(M)\neq{\rm Ker}f$, it follows that ${\rm rl}(\Omega M)=2$ and $l(\text{soc}(\Omega M))=l(\text{soc}P(M))=s$. Moreover, $l(\Omega M)=l(P(M))-l(M)=3s-t$,  so $l(\Omega M/\text{soc}(\Omega M))=2s-t$. That is, $\Omega M$ is of $(2s-t,s)$-type.

(2) The statement is dual to (1).

(3) The result  is clear.
\end{proof}

\begin{proposition}\label{4.4}
Let $M$ be an indecomposable $D(H_{\mathcal D})$-module of $(s,t)$-type.
\begin{enumerate}
\item[(1)] If $s<t$, then $t=s+1$ and $M\cong\Omega^{-s}V(l,\l)$ for some $1\<l\<n-1$ and $\l\in I_l$.
\item[(2)] If $s>t$, then $s=t+1$ and $M\cong\Omega^tV(l,\l)$  for some $1\<l\<n-1$ and $\l\in I_l$.
\end{enumerate}
\end{proposition}
\begin{proof}
%It is similar to \cite[Theorem 3.15]{Ch3}.

$(1)$ We first consider the case  $1=s<t$. Then $t=2$ by Lemma \ref{4.3}(1), and hence ${\rm head}M$ is simple. Let $f: P(M)\ra M$ be an epimorphism. Since $l(M)=3$ and $l(P(M))=4$, it follows  that $l({\rm Ker}f)=1$. Thus,  ${\rm Ker}f={\rm soc}P(M)\cong V(l,\l)$ for some $1\<l\<n-1$ and $\l\in I_l$.  Consequently,
$$M\cong P(M)/{\rm Ker}f=P(M)/{\rm soc}P(M)\cong\Omega^{-1}V(l,\l).$$
 Now assume $1<s<t$. Then $t<2s$ by Lemma \ref{4.3}. Set $r=t-s$,  so $1\<r<s$.
 There exists a unique positive integer $l'$ such that
 $$l'r<s\<(l'+1)r.$$
 By Lemma \ref{4.3}(1), one can show by the induction on $i$ that \\
 $\Omega^iM$ is of $(s-ir,s-(i-1)r)$-type for all $1\<i\<l'$.\\
  In particular,  $\Omega^{l'}M$ is of $(s-l'r,s-(l'-1)r)$-type.

If $s-l'r>1$, then by Lemma \ref{4.3}(1) we have
$$s-(l'-1)r<2(s-l'r),$$ which implies
 $(l'+1)r<s$, a contradiction. Hence $s-l'r=1$ and therefore $r=1$. It follows that $t=s+1$, $s=l'+1$, and $\Omega^{l'}M$ is of $(1,2)$-type. Thus, $\Omega^{s-1}M\cong \Omega^{-1}V(l,\l)$ for some $1\<l\<n-1$ and $\l\in I_l$, and so $M\cong \Omega^{-s}V(l,\l)$.

$(2)$ It is similar to (1) or dual to (1).
\end{proof}

\begin{lemma}\label{4.5}
For any $1\<l\<n-1$, $\l \in I_l$, we have the following Auslander-Reiten sequences in ${\rm mod}D(H_{\mathcal D})$:
\begin{enumerate}
\item[(1)] $0\rightarrow \Omega V(l,\l)\rightarrow V(n-l,\s(\l))\oplus V(n-l,\s^{-1}(\l))\oplus P(l, \l)\rightarrow \Omega^{-1}V(l,\l)\rightarrow 0$;
\item[(2)] $0\rightarrow \Omega^{t+2} V(l,\l)\rightarrow \Omega^{t+1}V(n-l,\s(\l))\oplus \Omega^{t+1}V(n-l,\s^{-1}(\l))\rightarrow \Omega^tV(l,\l)\rightarrow 0$;
\item[(3)] $0\rightarrow \Omega^{-t} V(l,\l)\rightarrow \Omega^{-(t+1)}V(n-l,\s(\l))\oplus \Omega^{-(t+1)}V(n-l,\s^{-1}(\l))\rightarrow \Omega^{-(t+2)}V(l,\l)\rightarrow 0$,
    \end{enumerate}
where $t\>0$ and $\Omega^0V(l,\l)=V(l,\l)$.
\end{lemma}
\begin{proof}
It is similar to \cite[Theorem 3.17]{Ch3}.
\end{proof}

In what follows, we investigate the indecomposable $D(H_{\mathcal D})$-modules of $(t,t)$-type.  For any $D(H_{\mathcal D})$-module $M$ and $\l \in \Lambda$,
define
$$M_{\l}:=\{v\in M|(g\g)v=\l(g\g)v, g\g\in G\times\G\}.$$
Then $M_{\l}$ is a subspace of $M$.

\begin{lemma}\label{4.6}
Let $M$ be a $D(H_{\mathcal D})$-module.
\begin{enumerate}
\item[(1)] $M=\oplus_{\l \in \L}M_{\l}$ as vector spaces.
\item[(2)]If $\mathcal{D}$ is of nilpotent type, then $x M_{\l}\subseteq M_{\l\phi}$ for any $\l \in \L$, and hence $M^{x}=\oplus_{\l \in \L}(M^{x}\cap M_{\l})$.
    \item[(3)] If $\mathcal{D}$ is of non-nilpotent type, then $\xi M_{\l}\subseteq M_{\l\phi^{-1}}$ for any $\l \in \L$, and hence $M^{\xi}=\oplus_{\l \in \L}(M^{\xi}\cap M_{\l})$.
\item[(4)] If $K$ and $N$ are submodules of $M$ such that $M=K\oplus N$, then
$M_{\l}=K_{\l}\oplus N_{\l}$ for all $\l \in \L$.
\item[(5)] If $f: M\ra N$ is a $D(H_{\mathcal D})$-module map, then $f(M_{\l})\subseteq N_{\l}$ for all $\l\in\L$. Furthermore, if $f$ is surjective, then $f(M_{\l})=N_{\l}$ for all $\l\in\L$.
\end{enumerate}
\end{lemma}

\begin{proof}
It follows from a straightforward verification.
\end{proof}

Now let $m$ be the integer defined in Lemma \ref{3.3}(3). To classify all indecomposable modules, we consider the two cases $m=1$ and $m>1$.

\subsection{The case of $m>1$}\label{m>1}
~

Throughout this subsection, we assume that $m>1$.

Let $1\<l\<n-1$ and $\l\in I_l$. Let $\{v_0, v_1, \cdots, v_{n-1}, u_0, u_1, \cdots, u_{n-1}\}$ be a standard basis of $P(l,\l)$ as defined previously. Define two subspaces $T_1(l,\l), \overline{T}_1(l,\l)\subset P(l,\l)$ as follows:
\begin{itemize}
\item If $\mathcal D$ is of nilpotent type:
$$T_1(l,\l)={\rm span}\{v_l, \cdots, v_{n-1}, u_{n-l}, \cdots, u_{n-1}\},\
\overline{T}_1(l,\l)={\rm span}\{u_0, u_1, \cdots, u_{n-1}\}$$
 \item If $\mathcal D$ is of non-nilpotent type:
$$T_1(l,\l)={\rm span}\{u_0, u_1, \cdots, u_{n-1}\},\ \overline{T}_1(l,\l)={\rm span}\{v_0, \cdots, v_{n-l-1}, u_0, \cdots, u_{l-1}\}.$$
\end{itemize}

By the proof of Proposition \ref{3.10}, both $T_1(l,\l)$ and $\ol{T}_1(l,\l)$ are submodules  of $P(l,\l)$ of $(1,1)$-type.  Moreover,  if $T$ is any submodules  of $P(l,\l)$ of $(1,1)$-type,  then $T=T_1(l,\l)$ or $T=\ol{T}_1(l,\l)$. Furthermore, the socle and head of these submodules are given by
$$\begin{array}{l}
{\rm soc}T_1(l,\l)\cong{\rm soc}\ol{T}_1(l,\l)\cong V(l,\l),\\
{\rm head}T_1(l,\l)\cong V(n-l,\s(\l)), \\
{\rm head}\ol{T}_1(l,\l)\cong V(n-l,\s^{-1}(\l)).\\
\end{array}$$
The above observation immediately yield the following lemma.

\begin{lemma}\label{4.7}
Let $1\<l,l'\<n-1$ and $\l\in I_l, \mu \in I_{l'}$.
\begin{enumerate}
\item[(1)] $T_1(l,\l)\ncong\ol{T}_1(l',\mu)$.
\item[(2)] $T_1(l,\l)\cong T_1(l',\mu)$ if and only if $l=l'$ and $\l=\mu$.
\item[(3)] $\ol{T}_1(l,\l)\cong\ol{T}_1(l',\mu)$ if and only $l=l'$ and $\l=\mu$.
\end{enumerate}
\end{lemma}

\begin{remark}\label{4.8}
Let $1\<l\<n-1$ and $\l\in I_l$. The following isomorphisms hold:
$$\begin{array}{ll}
\Omega^{-1}T_1(l,\l)\cong T_1(n-l,\s^{-1}(\l)),& \Omega T_1(l,\l)\cong T_1(n-l,\s(\l)),\\
\Omega^{-1}\ol{T}_1(l, \l)\cong\ol{T}_1(n-l,\s(\l)),& \Omega\ol{T}_1(l,\l)\cong\ol{T}_1(n-l,\s^{-1}(\l)).\\
\end{array}$$
Consequently,
$$\Omega^2T_1(l,\l)\cong T_1(l,\tau(\l))\ \mathrm{and}\  \Omega^2\ol{T}_1(l,\l)\cong\ol{T}_1(l,\tau^{-1}(\l)).$$
Moreover,  ${\rm End}_{D(H_{\mathcal D})}(T_1(l,\l))\cong\Bbbk$, ${\rm End}_{D(H_{\mathcal D})}(\ol{T}_1(l,\l))\cong\Bbbk$,  and for all $k\in\mathbb Z$ with $m\nmid k$,
$${\rm Hom}_{D(H_{\mathcal D})}(T_1(l,\l),T_1(l,\tau^k(\l)))=0,\ {\rm Hom}_{D(H_{\mathcal D})}(\ol{T}_1(l,\l),\ol{T}_1(l,\tau^k(\l)))=0.$$

\end{remark}

Similarly to \cite[Theorems 4.18]{SunChen}, one can establish the following two propositions.

\begin{proposition}\label{4.9}
Let $1\<l\<n-1$, $\l\in I_l$,  and $ t\in \mathbb{Z}$ with $t\>1$. Then  there exists an indecomposable $D(H_{\mathcal D})$-module $T_t(l,\l)$ of $(t,t)$-type. This module satisfies the following properties:
\begin{enumerate}
\item[(1)] $\Omega^2T_t(l,\l)\cong T_t(l,\tau(\l))$,  and \\
${\rm soc}T_t(l,\l)\cong\oplus_{k=0}^{t-1}V(l,\tau^{-k}(\l)),
{\rm head}T_t(l,\l)\cong\oplus_{k=0}^{t-1}V(n-l,\tau^{-k}(\s(\l)))$.
\item[(2)] For every $1\<j<t$, the module $T_t(l,\l)$ contains a unique submodule of $(j,j)$-type, which is isomorphic to $T_j(l,\l)$. Moreover,  the corresponding quotient module satisfies
$$ T_t(l,\l)/T_j(l,\l) \cong T_{t-j}(l,\tau^{-j}(\l)).$$
\item[(3)] For every $1\<j<t$, the unique submodule of $(j,j)$-type of $T_t(l,\l)$ is contained in the unique submodule of $(j+1,j+1)$-type.
\item[(4)] There exist AR-sequences:
\begin{eqnarray*}
&0\ra T_1(l,\l)\xrightarrow{f_1}
T_2(l,\l)\xrightarrow{g_1}
T_1(l,\tau^{-1}(\l))\ra 0,\\
&0\ra T_t(l,\l)\xrightarrow{\left(\begin{matrix}
                                       g_{t-1} \\
                                       f_t \\
                                     \end{matrix}\right)}
T_{t-1}(l,\tau^{-1}(\l))\oplus T_{t+1}(l,\l)\xrightarrow{(f'_{t-1}, g_t)}
T_t(l,\tau^{-1}(\l))\ra 0,
\end{eqnarray*}
where $t\>2$.
\end{enumerate}
\end{proposition}

\begin{proposition}\label{4.10}
 For any $1\<l\<n-1$, $\l\in I_l$ and $ t\in \mathbb{Z}$ with $t\>1$, there is an indecomposable $D(H_{\mathcal D})$-module $\ol{T}_t(l,\l)$ of $(t,t)$-type. We have the following properties:
\begin{enumerate}
\item[(1)] $\Omega^2\ol{T}_t(l,\l)\cong\ol{T}_t(l,\tau^{-1}(\l))$, ${\rm soc}\ol{T}_t(l,\l)\cong \oplus_{k=0}^{t-1}V(l,\tau^k(\l))$ and\\ ${\rm head}\ol{T}_t(l,\l)\cong \oplus_{k=0}^{t-1}V(n-l,\tau^{k-1}(\s(\l)))$.
\item[(2)] For every $1\<j<t$, $\ol{T}_t(l,\l)$ contains a unique submodule of $(j,j)$-type, which is isomorphic to $\ol{T}_j(l,\l)$ and the corresponding quotient module satisfies
 $$\ol{T}_t(l,\l)/\ol{T}_j(l,\l) \cong \ol{T}_{t-j}(l,\tau^j(\l)).$$
\item[(3)] For every $1\<j<t$, the unique submodule of $(j,j)$-type of $\ol{T}_t(l,\l)$ is contained in the unique submodule of $(j+1,j+1)$-type.
\item[(4)] There exist AR-sequences:
\begin{eqnarray*}
&0\ra\ol{T}_1(l,\l)\xrightarrow{f_1}\ol{T}_2(l,\l)\xrightarrow{g_1}
\ol{T}_1(l,\tau(\l))\ra 0,\\
&0\ra\ol{T}_{t}(l,\l)\xrightarrow{\left(\begin{matrix}
                                       g_{t-1} \\
                                       f_t \\
                                     \end{matrix}\right)}
\ol{T}_{t-1}(l,\tau(\l))\oplus\ol{T}_{t+1}(l,\l)\xrightarrow{(f'_{t-1}, g_t)}
\ol{T}_{t}(l,\tau(\l))\ra 0,
\end{eqnarray*}
where $t\>2$.
\end{enumerate}
\end{proposition}

\begin{corollary}\label{4.11}
Let $1\<l, l'\<n-1$, $\l\in I_l$, $\mu\in I_{l'}$ and $t, t'\in\mathbb{Z}$ with $t\>1$ and $t'\>1$. Then  the following isomorphisms hold:
\begin{enumerate}
\item[(1)] $T_t(l,\l)\ncong\ol{T}_{t'}(l',\mu)$.
\item[(2)] $T_t(l,\l)\cong T_{t'}(l',\mu)$ if and only if $t=t'$, $l=l'$ and $\l=\mu$.
\item[(3)] $\ol{T}_t(l,\l)\cong\ol{T}_{t'}(l',\mu)$ if and only if $t=t'$, $l=l'$ and $\l=\mu$.
\end{enumerate}
\end{corollary}
\begin{proof}
It follows from Lemma \ref{4.7}, Propositions \ref{4.9}(2) and \ref{4.10}(2) that
  $$l(T_t(l,\l))=2t\ \mathrm{and}\ l(\ol{T}_{t'}(l',\mu))=2t'.$$
\end{proof}

\begin{proposition}\label{4.12}
Let $M$ be an indecomposable $D(H_{\mathcal D})$-module of $(t,t)$-type. If $M$ contains a submodule of $(1,1)$-type, then $M\cong T_t(l,\l)$ or $M\cong\ol{T}_t(l,\l)$ for some $1\<l\<n-1$ and $\l \in I_l$.
\end{proposition}
\begin{proof}
If $t=1$, then $M$ is of $(1,1)$-type. Hence there exist $1\<l\<n-1$ and $\l\in I_l$ such that ${\rm soc}M\cong V(l,\l)$. Consequently,  $P(l,\l)$ serves as  an injective envelope of $M$. Thus, $M$ is isomorphic to a submodule of $P(l,\l)$. It follows that $M\cong T_1(l,\l)$ or $M\cong\ol{T}_1(l,\l)$.

Now let $t>1$ and assume that $M$ contains a submodule $N$ of $(1,1)$-type. Then as shown above, $N\cong T_1(l,\l)$ or $N\cong\ol{T}_1(l,\l)$ for some $1\<l\<n-1$ and $\l \in I_l$. By an argument similar to the proof of \cite[Theorem 4.16]{Ch3} we have $M\cong T_t(l,\l)$ or $M\cong\ol{T}_t(l,\l)$.
\end{proof}

Let $1\<l\<n-1$, $\l\in I_l$ and $P=\oplus_{k=0}^{m-1}P(l,\tau^k(\l))$.
 For any $0\<k\<m-1$, let
$\{v_0^k, v_1^k, \cdots, v_{n-1}^k, u_0^k, u_1^k, \cdots, u_{n-1}^k\}$ be a standard basis of $P(l,\tau^k(\l))$. Then $\{v^k_j,u^k_j|0\<j\<n-1, 0\<k\<m-1\}$ forms a basis of $P$.

Assume that $\mathcal{D}$ is of {\bf nilpotent type}.
Let $\eta\in{\Bbbk}^{\times}$. For all $0\<j\<n-1$ and $0\<k\<m-1$, define $x^k_j\in P$ by $x^k_j=u^k_j$ if $n-l\<j\<n-1$; $x^k_j=v^k_{j+l}+u^{k+1}_j$ if $0\<j\<n-l-1$ and $0\<k<m-1$; $x^{m-1}_j=v^{m-1}_{j+l}+\eta u^0_j$ if $0\<j\<n-l-1$ and $k=m-1$. A straightforward verification shows that
\begin{equation*}
\begin{array}{ll}
(g\g)x^k_j=\left\{
\begin{array}{ll}
(\phi^{j+l}\tau^k(\l))(g\g)x^k_j,& 0\<j\<n-l-1, 0\<k\<m-1,\\
(\phi^{j-n+l}\tau^k(\l))(g\g)x^k_j,& n-l\<j\<n-1, 0\<k\<m-1,\\
\end{array}\right.\\
\end{array}
\end{equation*}
\begin{equation*}
\begin{array}{ll}
\vspace{0.2 cm}
xx^k_j=\left\{
\begin{array}{ll}
x^k_{j+1},& 0\<j\<n-l-2, 0\<k\<m-1,\\
x^{k+1}_{n-l},& j=n-l-1, 0\<k<m-1,\\
\eta x^{0}_{n-l},& j=n-l-1, k=m-1,\\
x^k_{j+1},& n-l\<j\<n-2, 0\<k\<m-1,\\
0,& j=n-1, 0\<k\<m-1,\\
\end{array}\right.&\\
\xi x^k_j=\left\{
\begin{array}{ll}
x^k_{n-1}, & j=0, 0\<k\<m-1,\\
\a_j(n-l,\s(\l))x^k_{j-1},& 1\<j\<n-l-1, 0\<k\<m-1,\\
0, & j=n-l, 0\<k\<m-1,\\
\a_{j-n+l}(l,\l)x^k_{j-1},& n-l+1\<j\<n-1, 0\<k\<m-1,\\
\end{array}\right.\\
\end{array}
\end{equation*}
where $g\g\in G\times\G$.
It is easy to see that ${\rm span}\{x_j^k|0\<j\<n-1, 0\<k\<m-1\}$ is a submodule of rad$P$, denoted by $M_1(l,\l,\eta)$.
Clearly, $\{x_j^k|0\<j\<n-1, 0\<k\<m-1\}$ is a $\Bbbk$-basis of $M_1(l, \l,\eta)$. Such a basis is called a standard basis of $M_1(l,\l,\eta)$. Note that $\a_j(n-l,\s(\tau^k(\l)))=\a_j(n-l,\s(\l))$ and $\a_j(l,\tau^k(\l))=\a_j(l,\l)$, $j\>1$.

Assume that $\mathcal{D}$ is of {\bf non-nilpotent type}. Let $\eta\in{\Bbbk}^{\times}$. For any $0\<j\<n-1$ and $0\<k\<m-1$, define $x_j^k\in P$
by $x_j^k=u_j^k$ if $0\<j\<l-1$; $x_j^k=u_j^k+v_{j-l}^{k+1}$ if $l\<j\<n-1$ and $0\<k<m-1$;  $x_j^{m-1}=u_j^{m-1}+\eta v_{j-l}^0$ if $l\<j\<n-1$ and $k=m-1$.
A straightforward verification shows  that
\begin{eqnarray*}
\eqlabel{e2} &(g\g)x_j^k=(\phi^j\tau^k(\l))(g\g)x_j^k,\ g\g\in G\times\G, 0\<j\<n-1, 0\<k\<m-1,
\end{eqnarray*}
\begin{eqnarray*}
&x x_j^k=\left\{\begin{array}{ll}
\a_{j+1}(l,\l)x_{j+1}^k,& 0\<j\<l-2, 0\<k\<m-1,\\
0,& j=l-1, 0\<k\<m-1,\\
\a_{j+1-l}(n-l,\s(\l))x_{j+1}^k,& l\<j\<n-2, 0\<k\<m-1,\\
z_{l,\l}x_0^k+x_0^{k+1},& j=n-1, 0\<k<m-1,\\
z_{l,\l}x_0^{m-1}+\eta x_0^0,& j=n-1, k=m-1,\\
\end{array}\right.\\
&\xi x_j^k=\left\{\begin{array}{ll}
0,& j=0, 0\<k\<m-1,\\
x_{j-1}^k,& 1\<j\<n-1, 0\<k\<m-1.\\
\end{array}\right.
\end{eqnarray*}
This implies that ${\rm span}\{x_j^k|0\<j\<n-1, 0\<k\<m-1\}$ is a submodule of $P$, denoted by $M_1(l,\l,\eta)$.
Clearly, $\{x_j^k|0\<j\<n-1, 0\<k\<m-1\}$ is a $\Bbbk$-basis of $M_1(l,i,\eta)$. Such a basis is a {\it standard basis} of $M_1(l,\l,\eta)$. Note that $z_{l,\tau^k(\l)}=z_{l,\l}$ for any integer $k$.

\begin{lemma}\label{4.13}
Let $1\<l\<n-1$, $\l\in I_l$, and $\eta\in\Bbbk^{\times}$. We have the following:
\begin{enumerate}
\item[(1)] ${\rm rl}(M_1(l,\l,\eta))=2$, ${\rm soc}M_1(l,\l,\eta)\cong\oplus_{k=0}^{m-1}V(l,\tau^k(\l))$ and\\
  ${\rm head}M_1(l,\l,\eta)\cong\oplus_{k=0}^{m-1}V(n-l,\tau^k(\s(\l)))$.
\item[(2)] $M_1(l,\l,\eta)$ is an indecomposable module of $(m,m)$-type.
\item[(3)] $M_1(l,\l,\eta)$ does not contain any submodule of $(1,1)$-type.
\item[(4)] $M_1(l,\l,\eta)\ncong T_m(l',\mu)$ and $M_1(l,\l,\eta)\ncong\ol{T}_m(l',\mu)$ for any $1\<l'\<n-1$ and $\mu \in I_{l'}$.
\end{enumerate}
\end{lemma}

\begin{proof}
(1) Since $M_1(l,\l,\eta)\subseteq{\rm rad}P$, ${\rm rl}(M_1(l,\l,\eta))\<2$. However, ${\rm soc}M_1(l,\l,\eta)={\rm soc}P\cong \oplus_{k=0}^{m-1}V(l,\tau^k(\l))$, hence ${\rm rl}(M_1(l,\l,\eta))\neq 1$. Therefore,  ${\rm rl}(M_1(l,\l,\eta))=2$. Moreover, ${\rm head}M_1(l,\l,\eta)\cong\oplus_{k=0}^{m-1}V(n-l,\tau^k(\s(\l)))$.

(2) {\bf Case 1:} $\mathcal{D}$ is  of {\bf nilpotent type}. Assume that  $M_1(l,\l,\eta)=N\oplus K$ for some submodules $N$ and $K$ of $M_1(l,\l,\eta)$. Then for any $\mu\in \L$,  we have
$$M_1(l,\l,\eta)_{\mu}=N_{\mu}\oplus K_{\mu}.$$
 By the $D(H_{\mathcal D})$-module action, one can see that $x^k_j\in M_1(l,\l,\eta)_{\l\phi^{j+l+kn}}$ for $0\<j\<n-l-1$ and $0\<k\<m-1$,  $x^k_j\in M_1(l,\l,\eta)_{\l\phi^{j-n+l+kn}}$ for $n-l\<j\<n-1$ and $0\<k\<m-1$.

 Since ${\rm ord}(\phi)=mn$, dim$ M_1(l,\l,\eta)_{\l\phi^{j+kn}}=1$ for $0\<j\<n-1$ and $0\<k\<m-1$. Therefore,  each $x^k_j$ belongs either to $N$ or to $ K$. In particular, $x^0_0\in N$ or $x^0_0\in K$. Without loss of  generality, we may assume $x^0_0\in N$. Then $x^{n-1}x^0_0=x^1_{n-1}\in N$.

 If $x^1_0\in K$,  then $\xi x^1_0=x^1_{n-1}\in K$, which is a contradiction. Therefore $x^{1}_0\in N$. Similarly, one can verify that $x^2_0, \cdots, x^{m-1}_0\in N$. However, $M_1(l,\l,\eta)=\langle x^0_0, x^1_0, \cdots, x^{m-1}_0\rangle$, which implies that  $N=M_1(l,\l,\eta)$. Therefore,  $M_1(l,\l,\eta)$ is indecomposable. By the result in (1), it follows that $M_1(l,\l,\eta)$ is of $(m,m)$-type.

 {\bf Case 2:} $\mathcal{D}$ is  of {\bf non-nilpotent type}.
Suppose $M_1(l,\l,\eta)=N\oplus K$ for some submodules $N$ and $K$ of $M_1(l,\l,\eta)$. Then  for any $\mu\in \L$, we have
$$M_1(l,\l,\eta)_{\mu}=N_{\mu}\oplus K_{\mu}.$$
 By the module action, one sees that $x^k_j\in M_1(l,\l,\eta)_{\l\phi^{j+nk}}$ for $0\<j\<n-1$ and $0\<k\<m-1$. Moreover, dim$ M_1(l,\l,\eta)_{\l\phi^{j+nk}}=1$ for $0\<j\<n-1$,  $0\<k\<m-1$. Therefore,  each $x^k_j$ belongs either to $N$ or to $K$. In particular, $x^0_{n-1}\in N$ or $x^0_{n-1}\in K$. Without loss of  generality, we may assume $x^0_{n-1}\in N$. Then $xx^0_{n-1}=z_{l,\l}x^0_0+x^1_0\in N$ , however $\xi^{n-1}x^0_{n-1}=x^0_0\in N$. Consequently, $x^1_0\in N$. If $x^1_{n-1}\in K$ then $\xi^{n-1} x^1_{n-1}=x^1_{0}\in K$, a contradiction. Therefore, $x^{1}_{n-1}\in N$. Similarly, one can verify that $x^2_{n-1},x^3_{n-1},\cdots,x^{m-1}_{n-1}\in N$.

 However, $M_1(l,\l,\eta)=\langle x^0_{n-1},x^1_{n-1},\cdots , x^{m-1}_{n-1}\rangle$, which implies that $N=M_1(l,\l,\eta)$. Hence,  $M_1(l,\l,\eta)$ is indecomposable. By the result in  (1),  it follows that  $M_1(l,\l,\eta)$ is of $(m,m)$-type.

(3) Let $\pi: M_1(l,\l,\eta)\ra {\rm head}M_1(l,\l,\eta)$ be the canonical epimorphism.

{\bf Case 1:} Suppose that $\mathcal{D}$ is of {\bf nilpotent type}. Assume, for contradiction, that $M_1(l,\l,\eta)$ contains a submodule $N$ of $(1,1)$-type.
 By (1),  the image $\pi(N)$ is a simple submodule of ${\rm head}M_1(l,\l,\eta)$, so  $\pi(N)\cong  V(n-l,\tau^k(\s(\l)))$ for some $0\<k\<m-1$. Consequently,
$$\pi(N)=\oplus_{j=l}^{n-1}\pi(N)_{\l\phi^{kn+j}}\ \mathrm{and}\ {\rm dim}(\pi(N)_{\l\phi^{kn+j}})=1, \mbox{for all}\  l\<j\<n-1.$$
By Lemma \ref{4.6}(5), $\pi(N)_{\l\phi^{kn+j}}=\pi(N_{\l\phi^{kn+j}})$, implying $N_{\l\phi^{kn+j}}\neq 0$. From the argument in (2), we have
$$N_{\l\phi^{kn+j}}=M_1(l,\l,\eta)_{\l\phi^{kn+j}}=\Bbbk x^k_{j-l}, \ l\<j\<n-1,$$ so in particular, $x_0^k\in N$, and  thus $\langle x^k_0\rangle\subseteq N$.  But  ${\rm soc}\langle x^k_0\rangle\cong V(l,\tau^k(\l))\oplus V(l,\tau^{k+1}(\l))$, while  $l({\rm soc}N)=1$, a contradiction. Therefore, $M_1(l,\l,\eta)$  contains no submodule of $(1,1)$-type.

{\bf Case 2:} Suppose that $\mathcal{D}$ is of {\bf non-nilpotent type}.
Assume, for contradiction, that that $M_1(l,\l,\eta)$ contains a submodule $N$ of $(1,1)$-type.  By (1), the image $\pi(N)$ under the canonical projection  is a simple submodule of ${\rm head}M_1(l,\l,\eta)$,  so $\pi(N)\cong  V(n-l,\tau^k(\s(\l)))$ for some $0\<k\<m-1$. It follows that
$$\pi(N)=\oplus_{j=l}^{n-1}\pi(N)_{\l\phi^{kn+j}}\  \mathrm{and} \ {\rm dim}(\pi(N)_{\l\phi^{kn+j}})=1, \mbox{for  all}\   l\<j\<n-1.$$
By Lemma \ref{4.6}(5), $\pi(N)_{\l\phi^{kn+j}}=\pi(N_{\l\phi^{kn+j}})$, and hence $N_{\l\phi^{kn+j}}\neq 0$. By the proof of (2),  we have
$$N_{\l\phi^{kn+j}}=M_1(l,\l,\eta)_{\l\phi^{kn+j}}=\Bbbk x^k_j\ \mathrm{for}\  l\<j\<n-1,$$ so in particular, $x_{n-1}^k\in N$, and therefore $\langle x^k_{n-1}\rangle\subseteq N$. It is easy to see that ${\rm soc}\langle x^k_{n-1}\rangle\cong$$V(l,\tau^k(\l))\oplus V(l,\tau^{k+1}(\l))$, whereas $l({\rm soc}N)=1$, a contradiction. Therefore, $M_1(l,\l,\eta)$ does not contain any submodule of $(1,1)$-type.

(4) This follows from (3) and Propositions \ref{4.9}(2) and \ref{4.10}(2).
\end{proof}

\begin{lemma}\label{4.14}
Let $1\<l\<n-1$, $\l \in I_l$ and $\eta\in\Bbbk^{\times}$. Then:
\begin{enumerate}\eqlabel{e6}
\item[(1)]  $\O^{-1}M_1(l,\l,\eta)\cong M_1(n-l, \s^{-1}(\l), (-1)^m\eta).$
\item[(2)] $\O^{-2}M_1(l,\l,\eta)\cong M_1(l, \tau^{-1}(\l), \eta)$ and $\O^2M_1(l,\l,\eta)\cong M_1(l, \tau(\l), \eta)$.
\end{enumerate}
\end{lemma}

\begin{proof}
With the notations above, let $P=\oplus_{k=0}^{m-1}P(l,\tau^k(\l))$ denote the injective envelope of $M_1(l,\l,\eta)$, so that $M_1(l,\l,\eta)\subseteq P$. For any $y\in P$, let $\ol{y}$ denote its image  under the canonical epimorphism $P\ra P/M_1(l,\l,\eta)$.

 {\bf Case 1:}  Suppose that $\mathcal{D}$ is of {\bf nilpotent type}.
First, we have $\ol{v_{j+l}^k}=-\ol{u_j^{k+1}}$ for $0\<j\<n-l-1$ and $0\<k<m-1$, and $\ol{v_{j+l}^{m-1}}=-\eta\ol{u_j^0}$ for $0\<j\<n-l-1$.
Define elements $y^k_j\in P/M_1(l,\l,\eta)$ for $0\<j\<n-1$ and $0\<k\<m-1$ by
$$y_j^k:=\left\{ \begin{array}{ll}
(-1)^k\ol{v_j^k}, & \mathrm{if}\  0\<j\<l-1,\\
(-1)^k\ol{u_{j-l}^k}, &  \mathrm{if}\ l\<j\<n-1
\end{array}\right.
 $$
It is easy to see that  $\{y_j^k|0\<j\<n-1, 0\<k\<m-1\}$ forms a basis of $P/M_1(l,\l,\eta)$. A straightforward verification shows that for any $g\g\in G\times\G$,
\begin{equation*}
(g\g)y^k_j=\left\{
\begin{array}{ll}
(\phi^{j+n-l}\tau^k(\s^{-1}(\l)))(g\g)y^k_j,& 0\<j\<l-1, 0\<k\<m-1,\\
(\phi^{j-l}\tau^k(\s^{-1}(\l)))(g\g)y^k_j,& l\<j\<n-1, 0\<k\<m-1,\\
\end{array}\right.
\end{equation*}
\begin{equation*}
\begin{array}{ll}
\vspace{0.2 cm}
xy^k_j=\left\{
\begin{array}{ll}
y^k_{j+1},& 0\<j\<l-2, 0\<k\<m-1,\\
y^{k+1}_l,& j=l-1,0\<k<m-1,\\
(-1)^m\eta y^{0}_{l},& j=l-1,k=m-1,\\
y^k_{j+1},& l\<j\<n-2,0\<k\<m-1,\\
0,& j=n-1,0\<k\<m-1,\\
\end{array}\right.&\\
\vspace{0.2 cm}
\xi y^k_j=\left\{
\begin{array}{ll}
y^k_{n-1}, & j=0,0\<k\<m-1,\\
\a_{j}(l,\l)y^k_{j-1},& 1\<j\<l-1,0\<k\<m-1,\\
0, & j=l,0\<k\<m-1,\\
\a_{j-l}(n-l,\s^{-1}(\l))y^k_{j-1},& l+1\<j\<n-1,0\<k\<m-1.\\
\end{array}\right.\\
\end{array}
\end{equation*}

It follows that $\O^{-1}M_1(l,\l,\eta)\cong P/M_1(l,\l,\eta)\cong M_1(n-l,\s^{-1}(\l),(-1)^m\eta)$.

{\bf Case 2:} $\mathcal{D}$ is of {\bf non-nilpotent type}.
In $P/M_1(l,\l,\eta)$, we have $\ol{u^k_j}=-\ol{v^{k+1}_{j-l}}$ for $l\<j\<n-1$ and $0\<k<m-1$, and $\ol{u^{m-1}_j}=-\eta\ol{v^0_{j-l}}$ for $l\<j\<n-1$.
Now let $y_j^k:=(-1)^k\ol{v_j^k}$ for all $0\<j\<n-1$ and $0\<k\<m-1$. Then $\{y_j^k|0\<j\<n-1, 0\<k\<m-1\}$ is a basis of $P/M_1(l,i,\eta)$. A straightforward verification shows that
\begin{eqnarray*}
&(g\g)y_j^k=(\phi^{j-n+l}\tau^k(\l))(g\g)y_j^k,\ 0\<j\<n-1, 0\<k\<m-1,\\
&x y_j^k=\left\{\begin{array}{ll}
\a_{j+1}(n-l,\s^{-1}(\l))y_{j+1}^k,& 0\<j\<n-l-2, 0\<k\<m-1,\\
0,& j=n-l-1, 0\<k\<m-1,\\
\a_{j+1-n+l}(l,\l)y_{j+1}^k,& n-l\<j\<n-2, 0\<k\<m-1,\\
y_{l,\l}y_0^k+y_0^{k+1},& j=n-1, 0\<k<m-1,\\
y_{l,\l}y_0^{m-1}+(-1)^m\eta y_0^0,& j=n-1, k=m-1,\\
\end{array}\right.\\
&\xi y_j^k=\left\{\begin{array}{ll}
0,& j=0, 0\<k\<m-1,\\
y_{j-1}^k,& 1\<j\<n-1, 0\<k\<m-1.\\
\end{array}\right.
\end{eqnarray*}
Hence $\O^{-1}M_1(l,\l,\eta)\cong P/M_1(l,\l,\eta)\cong M_1(n-l,\s^{-1}(\l),(-1)^m\eta)$ by $y_{l,\l}=z_{n-l,\s^{-1}(\l)}$ for $0\<k\<m-1$.
This shows (1). (2) follows from (1).
\end{proof}

\begin{lemma}\label{4.15}
Let $1\<l\<n-1$, $\l\in I_l$ and $\eta\in\Bbbk^{\times}$. Then
$$M_1(l,\l,\eta)\cong M_1(l,\tau(\l),\eta).$$
\end{lemma}

\begin{proof}
Let $\{x_j^k|0\<j\<n-1, 0\<k\<m-1\}$ and $\{y_j^k|0\<j\<n-1, 0\<k\<m-1\}$ be the standard bases of $M_1(l,\l,\eta)$ and $M_1(l,\tau(\l),\eta)$, respectively. Define a linear map $f: M_1(l,\tau(\l),\eta)\ra M_1(l,\l,\eta)$ by
\begin{eqnarray*}
f(y_j^k)=\left\{\begin{array}{ll}
x_j^{k+1}, & 0\<j\<n-1, 0\<k<m-1,\\
\eta x_j^0, & 0\<j\<n-1, k=m-1.\\
\end{array}\right.
\end{eqnarray*}
Then $f$ is a linear isomorphism. By a straightforward verification, one can check that
for any $g\in G$, $\g\in \G$, $0\<j\<n-1$ and $0\<k\<m-1$,
\begin{eqnarray*}
f((g\g)y_j^k)=(g\g)f(y_j^k), f(xy_j^k)=xf(y_j^k), f(\xi y_j^k)=\xi f(y_j^k).
\end{eqnarray*}
Hence $f$ is a $D(H_{\mathcal D})$-module isomorphism.
This completes the proof.
\end{proof}

\begin{corollary}\label{4.16}
Let $1\<l\<n-1$, $\l\in I_l$ and $\eta\in\Bbbk^{\times}$. Then
\begin{eqnarray*}
\O^2M_1(l,\l,\eta)\cong M_1(l, \l, \eta)\cong\O^{-2}M_1(l,\l,\eta).
\end{eqnarray*}
\end{corollary}

\begin{proof}
It follows from Lemma \ref{4.14}(2) and Lemma \ref{4.15}.
\end{proof}

\begin{proposition}\label{4.17}
Let $1\<l, l'\<n-1$, $\l \in I_l, \l'\in I_{l'}$ and $\eta, \eta'\in\Bbbk^{\times}$. Then
$$M_1(l,\l,\eta)\cong M_1(l', \l', \eta') \Leftrightarrow l=l', \eta=\eta' \text{ and } \l'=\tau^k(\l)\text{ for some $k\in \mathbb{N}$ }.$$
\end{proposition}

\begin{proof}
If $l=l'$, $\eta=\eta'$ and $\l'=\tau^k(\l)\text{ for some $k\in \mathbb{N}$ }$, then by Lemma \ref{4.15},  we have
$$M_1(l,\l,\eta)\cong M_1(l', \l', \eta').$$
Conversely, assume that $M_1(l,\l,\eta)\cong M_1(l', \l', \eta')$. Then
$${\rm soc}M_1(l,\l,\eta)\cong{\rm soc}M_1(l', \l', \eta').$$
By Lemma \ref{4.13}(1), we have $V(l,\tau^k(\l))\cong V(l',\l')$ for some $0\<k\<m-1$. It follows from Proposition \ref{3.2}(1) or Proposition \ref{3.6} that
$l=l'$ and $\l'=\tau^k(\l)$. Hence,  by Lemma \ref{4.15}, $M_1(l',\l',\eta')\cong M_1(l, \l, \eta')$, and consequently $M_1(l,\l,\eta)\cong M_1(l, \l, \eta')$.

Let $f: M_1(l,\l,\eta)\ra M_1(l, \l, \eta')$ be a module isomorphism.
Denote by
$$\{x_j^k|0\<j\<n-1, 0\<k\<m-1\}\ \mathrm{and} \ \{y_j^k|0\<j\<n-1, 0\<k\<m-1\}$$
 the standard bases of $M_1(l,\l,\eta)$ and $M_1(l, \l, \eta')$, respectively.

If $\mathcal{D}$ is of {\bf nilpotent type}, then by
the proof of Lemma \ref{4.13}(2),
$$M_1(l,\l,\eta)_{\l\phi^{j+l+nk}}=\Bbbk x_j^k, \ \mathrm{and} \ M_1(l,\l,\eta')_{\l\phi^{j+l+nk}}=\Bbbk y_j^k$$
 for $0\<j\<n-l-1$ and $0\<k\<m-1$, and
 $$M_1(l,\l,\eta)_{\l\phi^{j-n+l+nk}}=\Bbbk x_j^k\ \mathrm{and} \ M_1(l,\l,\eta')_{\l\phi^{j-n+l+nk}}=\Bbbk y_j^k$$
 for $n-l\<j\<n-1$ and $0\<k\<m-1$.
Hence by Lemma \ref{4.6}(5),
there exist scalars $\b_0, \b_1, \cdots, \b_{m-1}\in\Bbbk^{\times}$ such that $$f(x_0^k)=\b_ky_0^k\ \mbox{for all}\ 0\<k\<m-1.$$
For any $1\<j\<n-l-1$ and $0\<k\<m-1$, we have
$$f(x_j^k)=f(x^jx^k_0)=x^jf(x^k_0)=x^j(\b_ky_0^k)=\b_ky^k_j.$$
In particular, $f(x^k_{n-l-1})=\b_ky^k_{n-l-1}$.   From $f(xx^k_{n-l-1})=xf(x^k_{n-l-1})$, it follows that $f(x^k_{n-l})=\b_{k-1}y^k_{n-l}$ for all $1\<k\<m-1$ and $f(x^0_{n-l})=\b_{m-1}\eta^{-1}\eta'y^0_{n-l}$. Therefore, for any $n-l\<j\<n-1$ and $0\<k\<m-1$,
$$f(x^k_j)=f(x^{j-n+l}x^k_{n-l})=x^{j-n+l}f(x^k_{n-l})=
\left\{
\begin{array}{ll}
\b_{k-1}y^k_j, & 1\<k\<m-1,\\
\b_{m-1}\eta^{-1}\eta'y^0_j, &  k=0.
\end{array}\right. $$

Finally, from $f(\xi x_0^k)=\xi f(x_0^k)$, we obtain $\b_k=\b_{k-1}$ for all $1\<k\<m-1$ and $\b_0=\b_{m-1}\eta^{-1}\eta'$. Consequently,  $\eta=\eta'$.

Now assume that $\mathcal{D}$ is of {\bf non-nilpotent type}.
By the proof of Lemma \ref{4.13}(2), we have
$$M_1(l,\l,\eta)_{\l\phi^{j+nk}}=\Bbbk x_j^k\ \mathrm{and}\ M_1(l,\l,\eta')_{\l\phi^{j+nk}}=\Bbbk y_j^k$$
 for all $0\<j\<n-1$ and $0\<k\<m-1$.  Hence,  by Lemma \ref{4.6}(5), there exist scalars $\b_0, \b_1, \cdots, \b_{m-1}\in\Bbbk^{\times}$ such that $f(x_{n-1}^k)=\b_ky_{n-1}^k$ for all $0\<k\<m-1$. For any $0\<j\<n-1$ and $0\<k\<m-1$, we have
$$f(x^k_j)=f({\xi}^{n-1-j}x_{n-1}^k)={\xi}^{n-1-j}f(x_{n-1}^k)
={\xi}^{n-1-j}(\b_ky_{n-1}^k)=\b_ky_j^k.$$
Moreover,  from $f(xx^k_{n-1})=xf(x^k_{n-1})$ for all $0\<k\<m-1$, it follows that
 $\b_0=\b_1=\cdots=\b_{m-1}$ and $\eta=\eta'$.
\end{proof}

\begin{corollary}\label{4.18}
Let $1\<l<n$, $\l \in I_l$ and $\eta\in\Bbbk^{\times}$. Then
$${\rm End}_{D(H_{\mathcal D})}(M_1(l,\l,\eta))\cong\Bbbk.$$
\end{corollary}

\begin{proof}
It follows from the proof of Proposition \ref{4.17}.
\end{proof}

\begin{lemma}\label{4.19}
Let $M$ be an indecomposable $D(H_{\mathcal D})$-module of $(t,t)$-type. Assume that $M$  contains no submodule of $(1,1)$-type. Then $M$ contains a submodule isomorphic to $M_1(l,\l,\eta)$ for some $1\<l<n$, $\l \in I_l$ and $\eta\in\Bbbk^{\times}$.
\end{lemma}

\begin{proof}
Clearly, $t>1$. By Proposition \ref{4.1}, we have
$$P=\oplus_{k=0}^{m-1}t_kP(l,\tau^k(\l))$$
as an injective envelope of $M$ for some $1\<l\<n-1$, $\l\in I_l$ and $t_k\in\mathbb N$ satisfying $\sum_{k=0}^{m-1}t_k=t$. Hence we may assume that
$$M\subseteq{\rm soc}^2P={\rm rad}P=\oplus_{k=0}^{m-1}t_k{\rm rad}P(l,\tau^k(\l)).$$
 It follows that ${\rm soc}M={\rm soc}P=\oplus_{k=0}^{m-1}t_k{\rm soc}P(l,\tau^k(\l))$. For each $1\<s\<t_k$, let
 $$\{v_0^{k,s}, \cdots, v_{n-1}^{k,s}, u_0^{k,s}, \cdots, u_{n-1}^{k,s}\}$$
  be a standard basis of the $s$-th copy  $P(l,\tau^k(\l))$ in the direct sum
  $t_kP(l,\tau^k(\l))$.

{\bf Case 1:} $\mathcal{D}$ is of {\bf nilpotent type}.
In this case, ${\rm rad}P$ has a $\Bbbk$-basis
$$\{v_l^{k,s}, \cdots, v_{n-1}^{k,s}, u_0^{k,s}, \cdots, u_{n-1}^{k,s}|0\<k\<m-1, 1\<s\<t_k\}$$
and
$${\rm soc}M={\rm soc}P={\rm span}\{u_{n-l}^{k,s}, \cdots, u_{n-1}^{k,s}|0\<k\<m-1,1\<s\<t_k\},$$ see the proof of Proposition \ref{3.10}.
Let
$$V={\rm span}\{v_l^{k,s}, \cdots, v_{n-1}^{k,s}, u_{0}^{k,s}, \cdots, u_{n-l-1}^{k,s}|0\<k\<m-1,1\<s\<t_k\}$$ and $U=M\cap V$. Then as vector spaces,
$${\rm rad}P=V\oplus{\rm soc}P\ \mathrm{and}\ M=U\oplus{\rm soc}M.$$
 For any $0\<j\<n-l-1$ and $0\<k\<m-1$,
 $$({\rm rad}P)_{\l\phi^{kn-n+l+j}}={\rm span}\{u_{j}^{k,s},v^{k-1,s'}_{j+l}|1\<s\<t_k,1\<s'\<t_{k-1}\}.$$
For any $n-l\<j\<n-1$ and $0\<k\<m-1$,
$$({\rm rad}P)_{\l\phi^{kn-n+l+j}}={\rm span}\{ u_{j}^{k,s}|1\<s\<t_k\},$$
 where we adopt the conventions $t_{-1}=t_{m-1}$ and $v_{j+l}^{-1,s'}=v_{j+l}^{m-1,s'}$.
Note that
$${\rm soc}P=\oplus_{j=n-l}^{n-1}\oplus_{k=0}^{m-1}({\rm rad}P)_{\l\phi^{kn-n+l+j}}\
\mathrm{and}\ V=\oplus_{j=0}^{n-l-1}\oplus_{k=0}^{m-1}({\rm rad}P)_{\l\phi^{kn-n+l+j}}.$$
Hence
$${\rm soc}M=\oplus_{j=n-l}^{n-1}\oplus_{k=0}^{m-1}M_{\l\phi^{kn-n+l+j}}\ \mathrm{and}\
U=\oplus_{j=0}^{n-l-1}\oplus_{k=0}^{m-1}M_{\l\phi^{kn-n+l+j}}.$$
 Moreover,
$M_{\l\phi^{kn-n+l+j}}=({\rm rad}P)_{\l\phi^{kn-n+l+j}}$ for all $n-l\<j\<n-1$ and $0\<k\<m-1$.
For any $0\<j\<n-1$, define $P_{[j]}=\oplus_{k=0}^{m-1}({\rm rad}P)_{\l\phi^{kn-n+l+j}}$ and $M_{[j]}=M\cap P_{[j]}$. Then
$${\rm rad}P=\oplus_{j=0}^{n-1}P_{[j]}, \ M_{[j]}=\oplus_{k=0}^{m-1}M_{\l\phi^{kn-n+l+j}}\ \mathrm{and}\  M=\oplus_{j=0}^{n-1}M_{[j]}.$$
Moreover,
$$M_{[j]}=P_{[j]}\ \mbox{for all}\  n-l\<j\<n-1, {\rm soc}M=\oplus_{j=n-l}^{n-1}M_{[j]}\ \mathrm{and}\ U=\oplus_{j=0}^{n-l-1}M_{[j]}.$$
By the structure of ${\rm rad}P$, the maps
$$M_{[j]}\ra M_{[j+1]}, v\mapsto xv\ \mathrm{and}\ M_{[j+1]}\ra M_{[j]}, v\mapsto \xi v$$
 are both bijective for any $0\<j<n-1$ with $j\neq n-l-1$.
It follows that ${\rm dim}M_{[j]}=t$ for all $0\<j\<n-1$ since $M$ is of $(t,t)$-type.

If $M^x\cap M_{[n-l-1]}\neq 0$, then by Lemma \ref{4.6}(2) and the equality $M_{[n-l-1]}=\oplus_{k=0}^{m-1}M_{\l\phi^{kn-1}}$ we have  $M^x\cap M_{\l\phi^{kn-1}}\neq 0$
for some $0\<k\<m-1$.  Since $M_{\l\phi^{kn-1}}\subseteq({\rm rad}P)_{\l\phi^{kn-1}}$, it follows from the action of $x$ on ${\rm rad}P$ that
$$M^x\cap M_{\l\phi^{kn-1}}\subseteq{\rm span}\{v_{n-1}^{k-1,s}|1\<s\<t_{k-1}\}.$$
 Let $0\neq w\in M^x\cap M_{\l\phi^{kn-1}}$. Then
$$w=\sum_{s=1}^{t_{k-1}}\b_sv_{n-1}^{k-1,s}\ \mbox{for some} \ \b_1,\cdots, \b_{t_{k-1}}\in\Bbbk.$$
A direct verification shows that the submodule $\langle w\rangle$ is isomorphic to ${T}_1(l,\tau^{k-1}(\l))$.  Hence,  $M$ contains a submodule of $(1,1)$-type, a contradiction. Therefore, $M^x\cap M_{[n-l-1]}=0$.

For any $0\<k\<m-1$, set
$$V_k^0={\rm span}\{v_{n-1}^{k-1, s}|1\<s\<t_{k-1}\}\ \mathrm{and}\ V_k^1={\rm span}\{u_{n-l-1}^{k, s}|1\<s\<t_k\}.$$
 Then
$$({\rm rad}P)_{\l\phi^{kn-1}}=V_k^0\oplus V_k^1\ \mathrm{and}\ M_{\l\phi^{kn-1}}\cap V_k^0=0.$$
Note that $V_k^0=0$ whenever $t_{k-1}=0$.
Hence $V_k^0\oplus M_{\l\phi^{kn-1}}\subseteq V_k^0\oplus V_k^1$ and  hence ${\rm dim} M_{\l\phi^{kn-1}}\<{\rm dim}V_k^1=t_k$.
However, since
$$\sum_{k=0}^{m-1}{\rm dim} M_{\l\phi^{kn-1}}={\rm dim}M_{[n-l-1]}=t=\sum_{k=0}^{m-1}t_k,$$ it follows that ${\rm dim} M_{\l\phi^{kn-1}}={\rm dim}V_k^1=t_k$ and hence $({\rm rad}P)_{\l\phi^{kn-1}}=V_k^0\oplus V_k^1=V_k^0\oplus M_{\l\phi^{kn-1}}$ for all $0\<k\<m-1$.

If one of $t_0, t_1, \cdots, t_{m-1}$ is  zero, we may assume $t_{k-1}=0$ but $t_k\neq 0$ for some $0\<k\<m-1$. Then
$$M_{\l\phi^{kn-1}}=V_k^1={\rm span}\{u_{n-l-1}^{k, s}|1\<s\<t_k\}.$$
In this case, one can check that the submodule $\langle u_{n-l-1}^{k,1}\rangle$ of $M$ is isomorphic to $\overline{T}_1(l,\tau^k(\l))$, a contradiction.  Therefore, $t_0, t_1, \cdots, t_{m-1}$ are all nonzero.

For each $0\<k\<m-1$, there exists a basis $\{x_{k,s}|1\<s\<t_k\}$ of $M_{\l\phi^{kn-1}}$ such that
$$x_{k,s}-u_{n-l-1}^{k,s}\in V_k^0\ \mbox{for all}\ 1\<s\<t_k.$$  We claim that the elements
$$x_{k,1}-u_{n-l-1}^{k,1}, \cdots, x_{k,t_k}-u_{n-l-1}^{k,t_k}$$
 are linearly independent over $\Bbbk$ for all $0\<k\<m-1$.

 Suppose, to the contrary, that $x_{k,1}-u_{n-l-1}^{k,1}, \cdots, x_{k,t_k}-u_{n-l-1}^{k,t_k}$ are linearly dependent over $\Bbbk$ for some $0\<k\<m-1$. If $t_k=1$, then $x_{k,1}-u_{n-l-1}^{k,1}=0$, which implies $u_{n-l-1}^{k,1}=x_{k,1}\in M$.  In this case, $M$ contains a submodule $\langle u_{n-l-1}^{k,1}\rangle$ of $(1,1)$-type,  contradicting our assumption. Hence $t_k>1$.

Without loss of generality, we may assume that
$$x_{k,t_k}-u_{n-l-1}^{k,t_k}=\sum_{j=1}^{t_k-1}\b_j(x_{k,j}-u_{n-l-1}^{k,j})\ \mbox{for some}\ \b_j\in\Bbbk.$$
 Then
 $$u_{n-l-1}^{k,t_k}-\sum_{j=1}^{t_k-1}\b_ju_{n-l-1}^{k,j}=x_{k,t_k}-\sum_{j=1}^{t_k-1}\b_jx_{k,j}\in M.$$
 Thus,  $M$ contains a submodule
 $$\langle u_{n-l-1}^{k,t_k}-\sum_{j=1}^{t_k-1}\b_ju_{n-l-1}^{k,j}\rangle$$
  of $(1,1)$-type as above, again a  contradiction. This proves the claim. Consequently,  $t_{k-1}={\rm dim}V_k^0\>t_k$ for all $0\<k\<m-1$, and hence
  $$t_0=t_1=\cdots=t_{m-1}.$$
  Therefore,  for  each $0\<k\<m-1$, there exists an invertible matrix $X_k\in M_{t_0}(\Bbbk)$ such that
$$(x_{k,1}-u_{n-l-1}^{k,1}, x_{k,2}-u_{n-l-1}^{k,2}, \cdots, x_{k,t_0}-u_{n-l-1}^{k,t_0})=(v_{n-1}^{k-1,1}, v_{n-1}^{k-1,2}, \cdots, v_{n-1}^{k-1,t_0})X_k.$$
It follows that $(x_{k,1}, \cdots, x_{k,t_0})X_k^{-1}=(u_{n-l-1}^{k,1}, \cdots, u_{n-l-1}^{k,t_0})X_k^{-1}+(v_{n-1}^{k-1,1}, \cdots, v_{n-1}^{k-1,t_0})$.

Let $X=X^{-1}_{m-1}\cdots X^{-1}_1X^{-1}_0$. Then $X$ is an invertible matrix in $M_{t_0}(\Bbbk)$. Since $\Bbbk$ is an algebraically closed field, there exists a nonzero vector $B=(\b_1, \cdots, \b_{t_0})^T\in\Bbbk^{t_0}$ and a nonzero scalar $\eta\in\Bbbk$ such that
$$XB=\eta B,$$
where $(\b_1, \cdots, \b_{t_0})^T$ denotes the transposition of $(\b_1, \cdots, \b_{t_0})$. For $0\<k\<m-1$, define $y_{n-l-1}^k\in M_{\l\phi^{kn-1}}$ by
$$y_{n-l-1}^k:=(x_{k,1}, x_{k,2}, \cdots, x_{k,t_0})X^{-1}_k\cdots X^{-1}_1X^{-1}_0B.$$
Then we have
$$\begin{array}{rl}
\vspace{0.1cm}
y_{n-l-1}^0=&(u_{n-l-1}^{0,1},\cdots,u_{n-l-1}^{0,t_0})X^{-1}_0B
+(v_{n-1}^{m-1,1},\cdots,v_{n-1}^{m-1,t_0})B,\\
y_{n-l-1}^k=&(u_{n-l-1}^{k,1},\cdots,u_{n-l-1}^{k,t_0})X^{-1}_k\cdots X^{-1}_1X^{-1}_0B\\
\vspace{0.1cm}
&+(v_{n-1}^{k-1,1},\cdots,v_{n-1}^{k-1,t_0})X^{-1}_{k-1}\cdots X^{-1}_0B, \ 1\<k\<m-1.\\
\end{array}$$
Note that the map $P_{[0]}\ra P_{[n-l-1]}, v \mapsto x^{n-l-1}v$ is a bijection and its restriction gives rise to a bijection from $M_{[0]}$ onto $M_{[n-l-1]}$. Hence if $v\in P_{[0]}$ satisfies $x^{n-l-1}v\in M_{[n-l-1]}$, then $v\in M_{[0]}$.
Define $y_0^k\in P_{[0]}$, $0\<k\<m-1$, by
$$\begin{array}{rl}
\vspace{0.1cm}
y_{0}^0=&(u_{0}^{0,1},\cdots,u_{0}^{0,t_0})X^{-1}_0B+(v_{l}^{m-1,1},\cdots,v_{l}^{m-1,t_0})B,\\
y_0^k=&(u_0^{k,1},\cdots,u_0^{k,t_0})X^{-1}_k\cdots X^{-1}_1X^{-1}_0B\\
&+(v_{l}^{k-1,1},\cdots,v_{l}^{k-1,t_0})X^{-1}_{k-1}\cdots X^{-1}_0B,\ 1\<k\<m-1.\\
\end{array}$$
Let $0\<k\<m-1$. It is easy to see that $x^{n-l-1}y_0^k=y_{n-l-1}^k$, and hence $y_0^k\in M_{[0]}$. Let $y_j^k=x^{j}y_0^k$ for $1\<j\<n-l-1$ and $0\<k\<m-1$. Then $y_j^k\in M_{\l\phi^{kn-n+l+j}}$ for all $0\<j\<n-l-1$. Moreover, $xy_j^k=y_{j+1}^k$ for all $0\<j<n-l-1$. Furthermore, for all $0\<j\<n-l-1$, we have
$$\begin{array}{rl}
\vspace{0.1cm}
y_j^0=&(u_j^{0,1},\cdots,u_j^{0,t_0})X^{-1}_0B+(v_{l+j}^{m-1,1},\cdots,v_{l+j}^{m-1,t_0})B,\\
y_j^k=&(u_j^{k,1},\cdots,u_j^{k,t_0})X^{-1}_k\cdots X^{-1}_1X^{-1}_0B\\
&+(v_{l+j}^{k-1,1},\cdots,v_{l+j}^{k-1,t_0})X^{-1}_{k-1}\cdots X^{-1}_0B,\ 1\<k\<m-1.\\
\end{array}$$
In particular, for all $0\<j\<n-l-1$, we have
$$y_j^{m-1}=(u_j^{m-1,1},\cdots,u_j^{m-1,t_0})\eta B+(v_{l+j}^{m-2,1},\cdots,v_{l+j}^{m-2,t_0})X^{-1}_{m-2}\cdots X^{-1}_0B.$$
For all $n-l\<j\<n-1$ and $0\<k\<m-1$, define $y^k_j\in{\rm rad}P$ by
$$y^k_j=\left\{\begin{array}{ll}
\vspace{0.1cm}
(u_j^{m-1,1},\cdots,u_j^{m-1,t_0})B,& \text{ if } k=0,\\
(u_j^{k-1,1},\cdots,u_j^{k-1,t_0})X^{-1}_{k-1}\cdots X^{-1}_0B, & \text{ if } 1\<k\<m-1.\\
\end{array}\right.$$
Clearly, the set $\{y_j^k|0\<j\<n-1, 0\<k\<m-1\}$ is linearly independent.
A straightforward verification shows that for any $g\g\in G\times\G$,
\begin{equation*}
(g\g)y^k_j=\left\{
\begin{array}{ll}
(\phi^{l+j}\tau^{k-1}(\l))(g\g)y^k_j,& 0\<j\<n-l-1, 0\<k\<m-1,\\
(\phi^{j-n+l}\tau^{k-1}(\l))(g\g)y^k_j,& n-l\<j\<n-1, 0\<k\<m-1,\\
\end{array}\right.\\
\end{equation*}
\begin{equation*}
xy^k_j=\left\{
\begin{array}{ll}
y^k_{j+1},& 0\<j\<n-l-2, 0\<k\<m-1,\\
y^{k+1}_{n-l},& j=n-l-1,0\<k<m-1\\
\eta y^{0}_{n-l},& j=n-l-1,k=m-1\\
y^k_{j+1},& n-l\<j\<n-2,0\<k\<m-1,\\
0,& j=n-1,0\<k\<m-1\\
\end{array}\right.\\
\end{equation*}
\begin{equation*}
\xi y^k_j=\left\{
\begin{array}{ll}
y^k_{n-1}, & j=0,0\<k\<m-1\\
\a_j(n-l,\s(\l))y^k_{j-1},& 1\<j\<n-l-1,0\<k\<m-1\\
0, & j=n-l,0\<k\<m-1\\
\a_{j-n+l}(l,\l)y^k_{j-1},& n-l+1\<j\<n-1,0\<k\<m-1\\
\end{array}\right.\\
\end{equation*}
It follows that $\{y_j^k|1\<j\<n, 0\<k\<m-1\}$ is a basis of the submodule $\langle y_0^0, y_0^1, \cdots, y_0^{m-1}\rangle$ of $M$ and $\langle y_0^0, y_0^1, \cdots, y_0^{m-1}\rangle\cong M_1(l,\tau^{-1}(\l),\eta)$.

{\bf Case 2:} $\mathcal{D}$ is of {\bf non-nilpotent type}.
In this case, ${\rm rad}P$ has a $\Bbbk$-basis
$$\{v_0^{k,s}, \cdots, v_{n-l-1}^{k,s}, u_0^{k,s}, \cdots, u_{n-1}^{k,s}|0\<k\<m-1, 1\<s\<t_k\}$$
and ${\rm soc}M={\rm soc}P={\rm span}\{u_{0}^{k,s}, \cdots, u_{l-1}^{k,s}|0\<k\<m-1,1\<s\<t_k\}$.
Let $V={\rm span}\{v_0^{k,s}, \cdots, v_{n-l-1}^{k,s}, u_{l}^{k,s}, \cdots, u_{n-1}^{k,s}|0\<k\<m-1,1\<s\<t_k\}$ and $U=M\cap V$.
As vector spaces,  we have
$${\rm rad}P=V\oplus{\rm soc}P\ \mathrm{and}\ M=U\oplus{\rm soc}M.$$
 For any $0\<j\<l-1$ and $0\<k\<m-1$,
 $$({\rm rad}P)_{\l\phi^{kn+j}}={\rm span}\{u_{j}^{k,s}|1\<s\<t_k\}.$$
For any $l\<j\<n-1$ and $0\<k\<m-1$,
$$({\rm rad}P)_{\l\phi^{kn+j}}={\rm span}\{ u_{j}^{k,s},v_{j-l}^{k+1,s'}|1\<s\<t_k,1\<s'\<t_{k+1}\},$$ where we set $t_{m}=t_{0}$ and $v_{j-l}^{m,s'}=v_{j-l}^{0,s'}.$
Note that
$${\rm soc}P=\oplus_{j=0}^{l-1}\oplus_{k=0}^{m-1}({\rm rad}P)_{\l\phi^{kn+j}}\
\mathrm{and}\ V=\oplus_{j=l}^{n-1}\oplus_{k=0}^{m-1}({\rm rad}P)_{\l\phi^{kn+j}}.$$
Hence
$${\rm soc}M=\oplus_{j=0}^{l-1}\oplus_{k=0}^{m-1}M_{\l\phi^{kn+j}}\ \mathrm{and}\ U=\oplus_{j=l}^{n-1}\oplus_{k=0}^{m-1}M_{\l\phi^{kn+j}}.$$  Moreover,
$M_{\l\phi^{kn+j}}=({\rm rad}P)_{\l\phi^{kn+j}}$ for all $0\<j\<l-1$ and $0\<k\<m-1$.

For any $0\<j\<n-1$, let $P_{[j]}=\oplus_{k=0}^{m-1}({\rm rad}P)_{\l\phi^{kn+j}}$ and $M_{[j]}=M\cap P_{[j]}$. Then
$${\rm rad}P=\oplus_{j=0}^{n-1}P_{[j]}, \ M_{[j]}=\oplus_{k=0}^{m-1}M_{\l\phi^{kn+j}}\ \mathrm{and}\  M=\oplus_{j=0}^{n-1}M_{[j]}.$$
Moreover, $M_{[j]}=P_{[j]}$ for all $0\<j\<l-1$, 
${\rm soc}M=\oplus_{j=0}^{l-1}M_{[j]}\ \mathrm{and}\ U=\oplus_{j=l}^{n-1}M_{[j]}.$
By the structure of ${\rm rad}P$, the maps
$$M_{[j]}\ra M_{[j+1]}, v\mapsto xv\ \mathrm{and}\ M_{[j+1]}\ra M_{[j]}, v\mapsto \xi v$$
are both bijective for any $0\<j<n-1$ with $j\neq l-1$.
It follows that ${\rm dim}M_{[j]}=t$ for all $0\<j\<n-1$ since $M$ is of $(t,t)$-type.

Suppose that $M^{\xi}\cap M_{[l]}\neq 0$.  Then by Lemma \ref{4.6}(3) and the decomposition
$M_{[l]}=\oplus_{k=0}^{m-1}M_{\l\phi^{kn+l}}$,  there exists some $0\<k\<m-1$ such that
$$M^{\xi}\cap M_{\l\phi^{kn+l}}\neq 0.$$
 Since $M_{\l\phi^{kn+l}}\subseteq({\rm rad}P)_{\l\phi^{kn+l}}$,  the action of $\xi$ on ${\rm rad}P$ implies that
 $$M^{\xi}\cap M_{\l\phi^{kn+l}}\subseteq{\rm span}\{v_{0}^{k+1,s}|1\<s\<t_{k+1}\}.$$
Let $0\neq w\in M^{\xi}\cap M_{\l\phi^{kn+l}}$. Then we can write
$w=\sum_{s=1}^{t_{k+1}}\b_sv_{0}^{k+1,s}$ for some $\b_1,\cdots, \b_{t_{k+1}}\in\Bbbk$. A straightforward verification shows that the submodule $\langle w\rangle$ is isomorphic to $\ol{T}_1(l,\tau^{k+1}(\l))$, which is a contradiction. Therefore, we must have  $M^{\xi}\cap M_{[l]}=0$.

For  each $0\<k\<m-1$,  set
$$V_k^0={\rm span}\{v_{0}^{k+1, s}|1\<s\<t_{k+1}\}\ \mathrm{and}\ V_k^1={\rm span}\{u_{l}^{k, s}|1\<s\<t_k\}.$$ Then
$$({\rm rad}P)_{\l\phi^{kn+l}}=V_k^0\oplus V_k^1\ \mathrm{and}\ M_{\l\phi^{kn+l}}\cap V_k^0=0.$$
Note that $V_k^0=0$ if $t_{k+1}=0$.  Hence we have
$$V_k^0\oplus M_{\l\phi^{kn+l}}\subseteq V_k^0\oplus V_k^1,$$
which implies ${\rm dim} M_{\l\phi^{kn+l}}\<{\rm dim}V_k^1=t_k$.
 However, 
 $$\sum_{k=0}^{m-1}{\rm dim} M_{\l\phi^{kn+l}}={\rm dim}M_{[l]}=t=\sum_{k=0}^{m-1}t_k.$$ It follows that ${\rm dim} M_{\l\phi^{kn+l}}={\rm dim}V_k^1=t_k$. Consequently, we have
 $$({\rm rad}P)_{\l\phi^{kn+l}}=V_k^0\oplus V_k^1=V_k^0\oplus M_{\l\phi^{kn+l}}$$
for all $0\<k\<m-1$.

If one of $t_0, t_1, \cdots, t_{m-1}$ is  zero, we may assume $t_{k+1}=0$ but $t_k\neq 0$ for some $0\<k\<m-1$. In this case,
$$M_{\l\phi^{kn+l}}=V_k^1={\rm span}\{u_{l}^{k, s}|1\<s\<t_k\},$$
 and one can check that the submodule $\langle u_{l}^{k,1}\rangle$ of $M$ is isomorphic to $T_1(l,\tau^k(\l))$, a contradiction. Hence $t_0, t_1, \cdots, t_{m-1}$ are all nonzero.

 Then,  for each $0\<k\<m-1$, there exists a basis $\{x_{k,s}|1\<s\<t_k\}$ of $M_{\l\phi^{kn+l}}$ such that $x_{k,s}-u_{l}^{k,s}\in V_k^0$ for all $1\<s\<t_k$. An argument similar to Case 1 shows that the vectors
 $$x_{k,1}-u_{l}^{k,1}, \cdots, x_{k,t_k}-u_{l}^{k,t_k}$$
  are linearly independent over $\Bbbk$ for all $0\<k\<m-1$. Consequently,
  $$t_{k+1}={\rm dim}V_k^0\>t_k\ \mbox{for all}\ 0\<k\<m-1,$$
  which implies $t_0=t_1=\cdots=t_{m-1}$. Therefore,  for each $0\<k\<m-1$, there exists an invertible matrix $X_k\in M_{t_0}(\Bbbk)$ such that
$$(x_{k,1}-u_{l}^{k,1}, x_{k,2}-u_{l}^{k,2}, \cdots, x_{k,t_0}-u_{l}^{k,t_0})=(v_{0}^{k+1,1}, v_{0}^{k+1,2}, \cdots, v_{0}^{k+1,t_0})X_k.$$
Then
$$(x_{k,1}, x_{k,2}, \cdots, x_{k,t_0})=(u_{l}^{k,1}, u_{l}^{k,2}, \cdots, u_{l}^{k,t_0})+(v_{0}^{k+1,1}, v_{0}^{k+1,2}, \cdots, v_{0}^{k+1,t_0})X_k.$$
Let $X:=X_{m-1}\cdots X_1X_0$. Then $X$ is an invertible matrix in $M_{t_0}(\Bbbk)$. Since $\Bbbk$ is an algebraically closed field, there is exists a nonzero vector $B=(\b_1, \cdots, \b_{t_0})^T\in\Bbbk^{t_0}$ and a nonzero scalar $\eta\in\Bbbk$ such that
$$XB=\eta B,$$
 For $0\<k\<m-1$, define   $y_{l}^k\in M_{\l\phi^{kn+l}}$ by
$$y_l^k:=\left\{\begin{array}{ll}
(x_{0,1},\cdots,x_{0,t_0})B, & \text{ if } k=0,\\
(x_{k,1},\cdots,x_{k,t_0})X_{k-1}\cdots X_1X_0B, & \text{ if } 1\<k\<m-1.\\
\end{array}\right.$$
Then, we have
$$\begin{array}{rl}
y_l^0=&(u_{l}^{0,1},\cdots,u_{l}^{0,t_0})B+(v_0^{1,1},\cdots,v_0^{1,t_0})X_0B,\\
y_l^k=&(u_l^{k,1},\cdots,u_l^{k,t_0})X_{k-1}\cdots X_1X_0B\\
&+(v_0^{k+1,1},\cdots,v_0^{k+1,t_0})X_k\cdots X_1X_0B, \ 1\<k\<m-1.\\
\end{array}$$
Note that the map $P_{[n-1]}\ra P_{[l]}, y\mapsto {\xi}^{n-l-1}y$ is a bijection and its restriction gives rise to a bijection from $M_{[n-1]}$ onto $M_{[l]}$. Hence,  if $y\in P_{[n-1]}$ satisfies ${\xi}^{n-l-1}y\in M_{[l]}$, then $y\in M_{[n-1]}$.
Define $y_{n-1}^k\in P_{[n-1]}$, $0\<k\<m-1$, by
$$\begin{array}{rl}
y_{n-1}^0=&(u_{n-1}^{0,1},\cdots,u_{n-1}^{0,t_0})B+(v_{n-l-1}^{1,1},\cdots,v_{n-l-1}^{1,t_0})X_0B,\\
y_{n-1}^k=&(u_{n-1}^{k,1},\cdots,u_{n-1}^{k,t_0})X_{k-1}\cdots X_1X_0B\\
&+(v_{n-l-1}^{k+1,1},\cdots,v_{n-l-1}^{k+1,t_0})X_k\cdots X_1X_0B, \ 1\<k\<m-1.\\
\end{array}$$
Let $0\<k\<m-1$. It is easy to see that ${\xi}^{n-l-1}y_{n-1}^k=y_l^k$, and hence $y_{n-1}^k\in M_{[n-1]}$. Let $y_j^k={\xi}^{n-1-j}y_{n-1}^k$ for $0\<j<n-1$. Then $y_j^k\in M_{\l\phi^{kn+j}}$ for all $0\<j\<n-1$. Moreover, we have
$$y^k_j=\left\{\begin{array}{ll}
\vspace{0.1cm}
(u_j^{0,1},\cdots,u_j^{0,t_0})B, & 0\<j\<l-1, k=0,\\
\vspace{0.1cm}
(u_j^{k,1},\cdots,u_j^{k,t_0})X_{k-1}\cdots X_1X_0B, & 0\<j\<l-1, 1\<j\<m-1, \\
\vspace{0.1cm}
(u_j^{0,1},\cdots,u_j^{0,t_0})B+(v_{j-l}^{1,1},\cdots,v_{j-l}^{1,t_0})X_0B, & l\<j\<n-1, k=0,\\
(u_j^{k,1},\cdots,u_j^{k,t_0})X_{k-1}\cdots X_1X_0B &\\
+(v_{j-l}^{k+1,1},\cdots,v_{j-l}^{k+1,t_0})X_k\cdots X_1X_0B, & l\<j\<n-1, 1\<k\<m-1.\\
\end{array}\right.$$
In particular, for any $l\<j\<n-1$, we have
$$y_j^{m-1}=(u_j^{m-1,1},\cdots,u_j^{m-1,t_0})X_{m-2}\cdots X_0B+\eta(v_{j-l}^{0,1},\cdots,v_{j-l}^{0,t_0})B.$$
Clearly, the set $\{y_j^k|0\<j\<n-1, 0\<k\<m-1\}$ is linearly independent.
By a straightforward verification, one can check that
\begin{eqnarray*}
&(g\g)y_j^k=(\phi^j\tau^k(\l))(g\g)y_j^k,\ \ \ g\g\in G\times\G, 0\<j\<n-1, 0\<k\<m-1,\\
&x y_j^k=\left\{\begin{array}{ll}
\a_{j+1}(l,\l)y_{j+1}^k,& 0\<j\<l-2, 0\<k\<m-1,\\
0,& j=l-1, 0\<k\<m-1,\\
\a_{j+1-l}(n-l,\s(\l))y_{j+1}^k,& l\<j\<n-2, 0\<k\<m-1,\\
z_{l,\l}y_0^k+y_0^{k+1},& j=n-1, 0\<k<m-1,\\
z_{l,\l}y_0^{m-1}+\eta y_0^0,& j=n-1, k=m-1,\\
\end{array}\right.\\
&\xi y_j^k=\left\{\begin{array}{ll}
0,& j=0, 0\<k\<m-1,\\
y_{j-1}^k,& 1\<j\<n-1, 0\<k\<m-1.\\
\end{array}\right.
\end{eqnarray*}
It follows that $\{y_j^k|0\<j\<n-1, 0\<k\<m-1\}$ is a basis of the submodule $\langle y_{n-1}^0, y_{n-1}^1, \cdots, y_{n-1}^{m-1}\rangle$ of $M$ and $\langle y_{n-1}^0, y_{n-1}^1, \cdots, y_{n-1}^{m-1}\rangle\cong M_1(l,\l,\eta)$.
\end{proof}

Similarly to \cite[Theorem 4.37]{SunChen}, one can prove the following proposition.

\begin{proposition}\label{4.20}
For any $\eta\in\Bbbk^{\times}$, $l, t\in \mathbb{Z}$ with $1\<l\<n-1$, $t\>1$, and $\l\in I_l$ there exists an indecomposable $D(H_{\mathcal D})$-module $M_t(l,\l,\eta)$ of $(tm,tm)$-type. Moreover,
\begin{enumerate}
\item[(1)] $\O^2M_t(l,\l,\eta)\cong M_t(l,\l,\eta)$,  ${\rm soc}M_t(l,\l,\eta)\cong\oplus_{k=0}^{m-1}tV(l,\tau^k(\l))$ and\\ ${\rm head}M_t(l,\l,\eta)\cong\oplus_{k=0}^{m-1}tV(n-l,\tau^k(\s(\l)))$.
\item[(2)] If $M_t(l,\l,\eta)$ contains a submodule of $(s,s)$-type, then $m|s$. Moreover, for any $1\<j<t$, $M_t(l,\l,\eta)$ contains a unique submodule of $(jm,jm)$-type, which is isomorphic to $M_j(l,\l,\eta)$ and the quotient module of $M_t(l,\l,\eta)$ modulo the submodule of $(jm,jm)$-type is isomorphic to $M_{t-j}(l,\l,\eta)$.
\item[(3)] For any $1\<j<t$, the unique submodule of $(jm,jm)$-type of $M_t(l,\l,\eta)$ is contained in that of $((j+1)m,(j+1)m)$-type.
\item[(4)] $M_t(l,\l,\eta)$ is not isomorphic to $T_{tm}(l',\l')$ or $\ol{T}_{tm}(l',\l')$ for any $1\<l'\<n-1$ and $\l'\in I_{l'}$.
\item[(5)] There exist Auslander-Reiten sequences:
\begin{eqnarray*}
&0\ra M_1(l,\l,\eta)\xrightarrow{f_1}
M_2(l,\l,\eta)\xrightarrow{g_1}
M_1(l,\l,\eta)\ra 0,\\
&0\ra M_t(l,\l,\eta)\xrightarrow{\left(\begin{matrix}
	g_{t-1} \\
	f_t \\
	\end{matrix}\right)}
M_{t-1}(l,\l,\eta)\oplus M_{t+1}(l,\l,\eta)\xrightarrow{(f'_{t-1}, g_t)}
M_t(l,\l,\eta)\ra 0\ (t\>2).
\end{eqnarray*}
\end{enumerate}
\end{proposition}

\begin{corollary}\label{4.21}
Let $\eta, \eta'\in\Bbbk^{\times}$ and $l, l', t, t'\in\mathbb Z$ with $1\<l, l'\<n-1$, $t, t'\>1$ and $\l\in I_l, \l'\in I_{l'}$. Then
$M_t(l,\l,\eta)\cong M_{t'}(l', \l', \eta') \Leftrightarrow t=t', l=l', \eta=\eta' \text{ and } \l'=\tau^k(\l)$ for some integer $k$
\end{corollary}

\begin{proof}
It follows from Proposition \ref{4.17}, Proposition \ref{4.20}(2,5), and the facts that $l(M_t(l,\l,\eta))=2tm$ and $l(M_{t'}(l', \l', \eta'))=2t'm$.
\end{proof}

\begin{proposition}\label{4.22}
Let $M$ be an indecomposable $D(H_{\mathcal D})$-module of $(t,t)$-type. If $M$ contains a submodule isomorphic to $M_1(l,\l,\eta)$ for some $1\<l\<n-1$, $\l\in I_{l}$ and $\eta\in\Bbbk^{\times}$, then $t=sm$ and $M\cong M_s(l,\l, \eta)$ for some $s\>1$.
\end{proposition}

\begin{proof}
Assume that $M$ contains a submodule isomorphic to $M_1(l,\l,\eta)$ for some $1\<l\<n-1$, $\l\in I_{l}$ and $\eta\in\Bbbk^{\times}$. Then $t\>m$. By Proposition \ref{4.12}, Lemma \ref{4.13}(4), and Propositions \ref{4.9}(2,3) and \ref{4.10}(2,3), one can verify that $M$ does not contain any submodule of $(1,1)$-type. Hence, by the proof of Lemma \ref{4.19}, $m$ divides $t$, so that   $t=sm$ for some $s\>1$. An argument similar to that in \cite[Theorem 4.16]{Ch3} then  shows that $M\cong M_t(l,\l,\eta)$.
\end{proof}

Summarizing the discussions in the last section and this section, we obtain the classification of finite dimensional indecomposable $D(H_{\mathcal D})$-modules as follows.

\begin{theorem}\label{4.23}
Assume $m>1$. A complete set of representatives of isomorphism classes of finite dimensional indecomposable $D(H_{\mathcal D})$-modules is given by
$$\left\{\begin{array}{c}
V(l',\l'), P(l,\l), \Omega^{\pm s}V(l,\l),\\
T_s(l,\l), \ol{T}_s(l,\l), M_s(l,\l,\eta)\\
\end{array}\left|\begin{array}{c}
1\<l\<n-1, 1\<l'\<n, s\>1,\\
\l \in I_l, \l'\in I_{l'}, \eta\in\Bbbk^{\times}\\
\end{array}\right.\right\}.$$
\end{theorem}

\subsection{The case of $m=1$}\label{m=1}
~

Throughout this subsection, unless otherwise stated, assume $m=1$. In this case, $\mathcal D$ is of nilpotent type by Remark \ref{3.4}. Let $\infty$ be a symbol with $\infty\notin\Bbbk$ and set $\ol{\Bbbk}:=\Bbbk\cup\{\infty\}$.

Let $1\<l\<n-1$ and $\l\in I_l$. Let $\{v_0, v_1, \cdots, v_{n-1}, u_0, u_1, \cdots, u_{n-1}\}$ be a standard basis of $P(l,\l)$ as given previousely. For any $\eta\in\ol{\Bbbk}$,  define a subspace $W_1(l,\l, \eta)\subset P(l,\l)$  by
$$W_1(l,\l,\eta):=\left\{ \begin{array}{ll}
{\rm span}\{u_0+\eta v_l, \cdots, u_{n-l-1}+\eta v_{n-1}, u_{n-l}, \cdots, u_{n-1}\}, & \eta\in\Bbbk, \\
{\rm span}\{v_l, \cdots, v_{n-1}, u_{n-l}, \cdots, u_{n-1}\}, & \eta=\infty.
\end{array}\right. $$
Since $m=1$, we have $\s^{-1}=\s$ and ${\rm ord}(\phi)=n$. One can easily check that $W_1(l,\l,\eta)$ is a submodule of $P(l,\l)$, with
$$ \mathrm{soc}W_1(l,\l,\eta)\cong V(l,\l)\ \mathrm{and}\  W_1(l,\l,\eta)/{\text{soc}}W_1(l,\l,\eta)\cong V(n-l,\s(\l))$$
 for any $\eta\in\ol{\Bbbk}$.
Hence, $W_1(l,\l,\eta)$ is an indecomposable module of $(1,1)$-type.

\begin{lemma}\label{4.24}
Let $1\<l\<n-1$, $\l\in I_l$ and $\eta\in\Bbbk$. Then $W_1(l,\l,\infty)\ncong W_1(l,\l,\eta)$.
\end{lemma}
\begin{proof}
Using the above notations, we have $P(l,\l)^x=\Bbbk v_{n-1}+\Bbbk u_{n-1}$. Hence
$$W_1(l,\l,\eta)^x\\=\Bbbk u_{n-1}\ \mathrm{and}\ W_1(l,\l,\infty)^x=\Bbbk v_{n-1}+\Bbbk u_{n-1}.$$
 This shows  that ${\rm dim}W_1(l,\l,\eta)^x=1$ and ${\rm dim}W_1(l,\l,\infty)^x=2$, and  therefore $W_1(l,\l,\infty)\ncong W_1(l,\l,\eta)$.
\end{proof}

\begin{lemma}\label{4.25}
Let $M$ be an indecomposable $D(H_{\mathcal D})$-module of $(1,1)$-type. Then $M$ is isomorphic to $W_{1}(l,\l,\eta)$ for some $1\<l\<n-1$, $\l\in I_l$ and $\eta\in \overline{\Bbbk}$.
\end{lemma}
\begin{proof}
Since $M$ is an indecomposable module of $(1,1)$-type, its socle ${\rm soc}M$ is simple but not projective. Hence, ${\rm soc}M\cong V(l,\l)$ for some $1\<l\<n-1$ and $\l\in I_l$, and  $P(l,\l)$ is an injective envelope of $M$. We may assume that $M$ is a submodule of $P(l,\l)$. Then ${\rm soc}M={\rm soc}P(l,\l)$, $M/{\rm soc}M\subset{\rm soc}^2P(l,\l)/{\rm soc}P(l,\l)$ and $M\subset{\rm soc}^2P(l,\l)$.

By the proof of Proposition \ref{3.10}, ${\rm soc}^2P(l,\l)/{\rm soc}P(l,\l)\cong 2V(n-l,\s(\l))$. Hence,  ${\rm head}M\cong V(n-l,\s(\l))$, and so
$${\rm head}M=\oplus_{j=0}^{n-l-1}({\rm head}M)_{\s(\l)\phi^j}=\oplus_{j=l}^{n-1}({\rm head}M)_{\l\phi^j}.$$
Moreover, ${\rm dim}(({\rm head}M)_{\l\phi^j})=1$ for all $l\<j\<n-1$.
By Lemma \ref{4.6}(5), $({\rm head}M)_{\l\phi^l}=\pi(M_{\l\phi^l})$ and $({\rm soc}^2P(l,\l)/{\rm soc}P(l,\l))_{\l\phi^l}=\pi(({\rm soc}^2P(l,\l))_{\l\phi^l})$,
where $\pi: P(l,\l)\ra P(l,\l)/{\rm soc}P(l,\l)$ is the canonical epimorphism.

Let $\{v_0,v_2,\cdots, v_{n-1}, u_0,u_1,\cdots,u_{n-1}\}$ be the standard basis of $P(l,\l)$. Then by the proof of Proposition \ref{3.10}, we have
$${\rm soc}^2P(l,\l)={\rm span}\{v_l, \cdots, v_{n-1}, u_0, u_1, \cdots, u_{n-1}\}.$$
 Hence $({\rm soc}^2P(l,\l))_{\l\phi^l}=\Bbbk v_l+\Bbbk u_0$. Since $0\neq M_{\l\phi^l}\subseteq({\rm soc}^2P(l,\l))_{\l\phi^l}$, there exist scales $\a, \b\in\Bbbk$ such that
  $$0\neq\a u_0+\b v_l\in M_{\l\phi^l}.$$
  If $\a=0$, then $v_l\in M$, and hence $\langle v_l\rangle\subseteq M$. However, \\
  $\langle v_l\rangle={\rm span}\{v_l,\cdots, v_{n-1}, u_{n-l}, \cdots, u_{n-1}\}$ is a module of $(1,1)$-type. Thus,  $\langle v_l\rangle=M$, and so $M\cong W_1(l,\l,\infty)$.

If $\a\neq 0$, then $u_0+\eta v_l\in M$, where $\eta=\a^{-1}\b\in\Bbbk$. In this case, a similar argument  shows that $M\cong W_1(l,\l,\eta)$.
\end{proof}

\begin{proposition}\label{4.26}
Let $1\<l,l'\<n-1$, $\l\in I_l$, $\l'\in I_{l'}$ and $\eta, \eta'\in\ol{\Bbbk}$. Then $W_{1}(l,\l,\eta)\cong W_{1}(l',\l',\eta')$ if and only if $l=l'$, $\l=\l'$ and $\eta=\eta'$.
\end{proposition}
\begin{proof}
If $l=l'$, $\l=\l'$ and $\eta=\eta'$, then clearly, $W_{1}(l,\l,\eta)\cong W_{1}(l',\l',\eta')$.

Conversely, assume $W_{1}(l,\l,\eta)\cong W_{1}(l',\l',\eta')$. Then ${\rm soc}W_{1}(l,\l,\eta)\cong{\rm soc}W_{1}(l',\l',\eta')$, and hence $V(l,\l)\cong V(l',\l')$. By Proposition \ref{3.2}(1), $l=l'$ and $\l=\l'$, and so
$$W_{1}(l,\l,\eta)\cong W_{1}(l,\l,\eta').$$
 By Lemma \ref{4.24}, it is enough to consider the case where $\eta$ and $\eta'$ are contained in $\Bbbk$.
Let $f: W_{1}(l,\l,\eta)\ra W_{1}(l,\l,\eta')$ be a module isomorphism. Then
$$f(W_{1}(l,\l,\eta)_{\l\phi^j})=W_{1}(l,\l,\eta')_{\l\phi^j} \
\mbox{for all}\ 0\<j\<n-1.$$  Using the bases of $W_{1}(l,\l,\eta)$ and $W_{1}(l,\l,\eta')$ given before, we have
$$W_{1}(l,\l,\eta)_{\l\phi^l}=\Bbbk(u_0+\eta v_l), \ W_{1}(l,\l,\eta')_{\l\phi^l}=\Bbbk(u_0+\eta'v_l).$$
 Hence, there exists  some $\a\in\Bbbk^{\times}$  such that
 $$f(u_0+\eta v_l)=\a(u_0+\eta' v_l).$$
 Applying $x^{n-1}$, we obtain
 $$f(x^{n-1}(u_0+\eta v_l))=x^{n-1}f(u_0+\eta v_l),$$
 which implies that $f(u_{n-1})=\a u_{n-1}$.

 Similarly, applyinh $\xi$, we obtain
$$f(\xi(u_0+\eta v_l))=\xi f(u_0+\eta v_l),$$
which yields  $\eta\a u_{n-1}=\eta'\a u_{n-1}$. This implies $\eta=\eta'$.
\end{proof}

\begin{remark}\label{4.27}
Let $1\<l\<n-1$, $\l\in I_l$ and $\eta\in\ol{\Bbbk}$. Then we have
$$
\Omega^{-1}W_1(l,\l,\eta)\cong W_1(n-l,\s(\l),-\eta), \ \Omega W_1(l,\l,\eta)\cong W_1(n-l,\s(\l),-\eta),$$
where we regard $-\infty=\infty$. Hence $\Omega^2W_1(l,\l,\eta)\cong W_1(l,\l,\eta)$. Moreover, one can check that ${\rm End}_{D(H_{\mathcal D})}(W_1(l,\l,\eta))\cong\Bbbk$ and
${\rm Hom}_{D(H_{\mathcal D})}(W_1(l,\l,\eta),W_1(l,\l,\eta')))=0$ for any $\eta'\in\ol{\Bbbk}$ with $\eta'\neq\eta$.
\end{remark}

Similarly to \cite[Theorem 4.18]{SunChen}, one can prove the following proposition.

\begin{proposition}\label{4.28}
For any $1\<l\<n-1$, $\l\in I_l$, $\eta \in \overline{\Bbbk}$ and $ t\in \mathbb{Z}$ with $t\>1$, there is an indecomposable $D(H_{\mathcal D})$-module $W_t(l,\l,\eta)$ of $(t,t)$-type. We have the following properties:
\begin{enumerate}
\item[(1)] $\Omega^2W_t(l,\l,\eta)\cong W_t(l,\l,\eta)$,  ${\rm soc}W_t(l,\l,\eta)\cong tV(l,\l)$ and\\
${\rm head}W_t(l,\l,\eta)\cong tV(n-l,\s(\l))$.
\item[(2)] For any $1\<j<t$, $W_t(l,\l,\eta)$ contains a unique submodule of $(j,j)$-type, which is isomorphic to $W_j(l,\l,\eta)$ and the quotient module of $W_t(l,\l,\eta)$ modulo the submodule of $(j,j)$-type is isomorphic to $W_{t-j}(l,\l,\eta)$.
\item[(3)] For any $1\<j<t$, the unique submodule of $(j,j)$-type of $W_t(l,\l,\eta)$ is contained in that of $(j+1,j+1)$-type.
\item[(4)] There are AR-sequences
\begin{eqnarray*}
&0\ra W_1(l,\l,\eta)\xrightarrow{f_1}
W_2(l,\l,\eta)\xrightarrow{g_1}
W_1(l,\l,\eta)\ra 0,\\
&0\ra W_t(l,\l,\eta)\xrightarrow{\left(\begin{matrix}
                                       g_{t-1} \\
                                       f_t \\
                                     \end{matrix}\right)}
W_{t-1}(l,\l,\eta)\oplus W_{t+1}(l,\l,\eta)\xrightarrow{(f'_{t-1}, g_t)}
W_t(l,\l,\eta)\ra 0
\end{eqnarray*}
for $t\>2$.
\end{enumerate}
\end{proposition}

\begin{proposition}\label{4.29}
Let $M$ be an indecomposable left $D(H_{\mathcal D})$-module of $(t,t)$-type. Then $M$ contains a submodule of $(1,1)$-type, and consequently $M\cong W_t(l,\l,\eta)$ for some $1\<l\<n-1$, $\l \in I_l$ and $\eta \in \overline{\Bbbk}$.
\end{proposition}
\begin{proof}
First, by an argument similar to,  but much easier than,  the proof of Lemma \ref{4.19}, one can check that $M$ contain a submodule isomorphic to $W_1(l,\l,\eta)$ for some $1\<l\<n-1$, $\l\in I_l$ and $\eta\in\ol{\Bbbk}$. Then, similarly to \cite[Theorem 4.16]{Ch3}, one can show that $M\cong W_t(l,\l,\eta)$.
\end{proof}

Summarizing the above discussion, we obtain the classification of finite dimensional indecomposable modules over $D(H_{\mathcal D})$ as follows.

\begin{theorem}\label{4.30}
Assume $m=1$. A complete set of representatives  of isomorphism classes of finite dimensional indecomposable $D(H_{\mathcal D})$-modules is given by
$$\left\{\begin{array}{c}
V(l',\l'), P(l,\l),\\ \Omega^{\pm s}V(l,\l),
 W_s(l,\l,\eta)\\
\end{array}\left|\begin{array}{c}
1\<l\<n-1, 1\<l'\<n, s\>1,\\
\l \in I_l, \l'\in I_{l'}, \eta\in\overline{\Bbbk}.\\
\end{array}\right.\right\}.$$

\end{theorem}

By Theorem \ref{4.23} and Theorem \ref{4.30}, we  immediately obtain the following corollary.

\begin{corollary}\label{4.31}
In all cases, where $m>1$ or $m=1$,  the Drinfeld double $D(H_{\mathcal D})$ is of tame representation type.
\end{corollary}

\vskip 1cm

\centerline{ACKNOWLEDGMENTS}

This work is supported by NNSF of China (Nos. 12201545, 12071412).\\

\end{document}